\documentclass[11pt,oneside]{amsart}
\calclayout

\usepackage{amsmath}
\usepackage{amssymb,amsthm}
\usepackage{graphicx}
\usepackage[colorlinks=true]{hyperref}
\usepackage{xurl}
\hypersetup{breaklinks=true}

\usepackage{color,xcolor}
\usepackage{mathrsfs}
\usepackage[all]{xy}
\usepackage{mathabx} 
\usepackage{mathtools}
\usepackage{mathbbol}

\usepackage{tikz}

\usepackage[backend=biber,style=alphabetic,maxalphanames=4,url=false,eprint=false]{biblatex}
\newtheorem{theorem}{Theorem}[section]
\newtheorem{lemma}[theorem]{Lemma}
\newtheorem{proposition}[theorem]{Proposition}
\newtheorem{corollary}[theorem]{Corollary}

\theoremstyle{definition}
\newtheorem{defn}[theorem]{Definition}

\theoremstyle{remark}
\newtheorem{remark}[theorem]{Remark}

\newtheorem*{claim}{Claim}

\newtheorem{question}[theorem]{Question}

\def\Z{\mathbb Z}
\def\R{\mathbb R}
\def\N{\mathbb N}
\def\RH{\mathbb H}

\def\G{\mathcal G}
\def\C{\mathcal C}

\def\XX{\mathbf X}
\def\T{\mathcal T}
\def\L{\mathcal L}

\def\H{\mathcal H}

\def\P{\mathcal P}
\def\XX{\mathbf X}

\def\Isom{\mathrm{Isom}}

\def\i{\mathrm{i}}

\def\rk{\mathrm{rk}}
\def\gl{\mathrm{GL}}
\def\sl{\mathrm{SL}}

\def\out{\mathrm{Out}}
\def\aut{\mathrm{Aut}}
\def\mcg{\mathrm{MCG}}
\def\pmcg{\mathrm{PMCG}}
\def\stab{\mathrm{Stab}}

\newcommand{\orth}{\perp}

\newcommand{\tsh}[1]{\left\{\kern-.7ex\left\{#1\right\}\kern-.7ex\right\}}

\title{Central extensions and proper actions on products of hyperbolic spaces}

\author{Bingxue Tao}
\address{Department of Mathematics, Kyoto University, Kyoto 6068502, Japan}
\email{tao.bingxue.s66@kyoto-u.jp}

\author{Renxing Wan}
\address{School of Mathematical Sciences,  Key Laboratory of MEA (Ministry of Education) \& Shanghai Key Laboratory of PMMP,  East China Normal University, Shanghai 200241, China P. R.}
\email{rxwan@math.ecnu.edu.cn}


\begin{document}

\begin{abstract}
    The main result of this paper identifies boundedness of the Euler class as the exact obstruction to preserving property QT under central extensions. For a central extension of groups $1\to Z\to E\to G\to 1$, we prove that $E$ has property QT if and only if $Z$ is finitely generated, $G$ has property QT, and the Euler class of the extension is bounded. This is achieved by using quasimorphisms as a bridge between central extensions and group actions. 
    As applications, we show that mapping class groups of finite-type surfaces possibly with boundary, multicurve stabilizers, and outer automorphism groups of torsion-free one-ended hyperbolic groups have property QT. We also show that Sela’s central extension description of the latter has a bounded Euler class.

    In addition, we introduce property PH, which is a weaker analogue of property QT related to locally uniform exponential growth of groups, and derive the same stability results under central extensions. We provide several examples with or without property PH. In particular, the fundamental group of a compact orientable $3$-manifold $M$ has property PH whenever no summand in the sphere-disk decomposition  of $M$ supports Nil geometry. 
\end{abstract}
\maketitle

\section{Introduction}

\subsection{Central extensions and property QT}

The central theme of this paper is the stability of properties defined by group actions on finite products of hyperbolic spaces under central extensions. Let
\begin{equation}\label{Extension}
    1\to Z\xrightarrow{\iota}E\xrightarrow{\pi}G\to 1.
\end{equation}
be a central extension of groups. When $Z$ is finitely generated, such an extension defines a cohomology class $[\omega]\in H^2(G,Z)$, called the \textit{Euler class} (see Section \ref{subsec: CentralExtension}).

The two main properties we study are \textit{property QT} and \textit{property PH}. In \cite{BBF21}, Bestvina--Bromberg--Fujiwara proposed property QT for finitely generated groups to apply the projection complex techniques developed in \cite{BBF15}. Since then, the study of property QT has received a great deal of interest; see \cite{HP22,NY23a,HNY25,But25,Ver24,Tao26,PSZ25} for example. Specifically, a finitely generated group $G$ has \textit{property QT} if $G$ acts isometrically on a finite product of quasi-trees equipped with the $\ell^1$-metric such that the orbit map is a quasi-isometric embedding. In this paper, when we say a group has property QT, we always assume that the group is finitely generated.

In this paper, we also introduce a new property for groups, called \textit{property PH} (standing for proper actions and hyperbolic spaces). A (not necessarily finitely generated) group $G$ has \textit{property PH} if $G$ acts properly by isometries on a finite product of hyperbolic spaces equipped with the $\ell^1$-metric and the stabilizer of each factor space acts coboundedly on this factor space (see Definition \ref{Def: PH}). For finitely generated groups, property PH is a generalization of property QT (see Lemma \ref{Lem: QTImpPH}), and their relationship is described in detail in Subsection \ref{subsec: PHintro} below.

Our main result characterizes when these properties are preserved under a central extension:

\begin{theorem}\label{MainThm1}
    For the central extension (\ref{Extension}), $E$ has property PH (resp. QT) if and only if $Z$ is finitely generated, $G$ has property PH (resp. QT), and the Euler class is bounded.
\end{theorem}

A key step to prove Theorem \ref{MainThm1} is to characterize central extensions with bounded Euler classes using quasimorphisms and lineal actions (see Proposition \ref{Prop: bddEu}). Briefly speaking, for a central extension with finitely generated kernel, boundedness of the Euler class is equivalent to quasi-splitting, quasi-retraction, virtual quasi-splitting, and the existence of orientable lineal actions detecting the central subgroup. Shortly before we finished this paper, part of this characterization was also obtained by Fournier-Facio, Mangioni, and Sisto independently (see \cite[Proposition 2.9]{FMS25}). Their main result is a parallel result of Theorem \ref{MainThm1} on hierarchical hyperbolicity.

Since any central extension of a hyperbolic group by a finitely generated abelian group has a bounded Euler class \cite[Theorem 3.1]{NR97}, it follows directly from Theorem \ref{MainThm1} that such a central extension has property PH. Moreover, in \cite{BBF21}, Bestvina--Bromberg--Fujiwara show that residually finite hyperbolic groups have property QT. Thus, we have the following theorem.

\begin{theorem}
    Any central extension of a hyperbolic group $G$ by a finitely generated abelian group has property PH. If moreover, $G$ is residually finite, then any central extension of $G$ by a finitely generated abelian group has property QT.
\end{theorem}

This result may also be proved by following the proof of \cite[Corollary 4.3]{HRSS24}.
We remind the readers that it is still an open conjecture that all hyperbolic groups are residually finite.

Let $\Sigma$ be a \emph{finite-type} surface, i.e., a closed surface with finitely many punctures.
Bestvina--Bromberg--Fujiwara \cite{BBF21} shows that the mapping class group $\mcg(\Sigma)$ has property QT. Since the capping sequences of mapping class groups are central extensions with bounded Euler classes, we deduce from Theorem \ref{MainThm1} that

\begin{theorem}\label{IntroThm: MCG}
    $\mcg(\Sigma)$ has property QT for any finite-type surface $\Sigma$ with boundary. In particular, braid groups have property QT.
\end{theorem}

With more effort, we also show that

\begin{theorem}\label{IntroThm: Multicurve}
    Let $\Sigma$ be a finite-type surface possibly with boundary. The stabilizer in $\mcg(\Sigma)$ of any multicurve on $\Sigma$ has property QT.
\end{theorem}

Another application is about the outer automorphism group $\out(G)$ of an one-ended torsion-free hyperbolic group $G$.
We prove that

\begin{theorem}\label{IntroThm: Out(G)}
    For any one-ended torsion-free hyperbolic group $G$, $\out(G)$ has property QT.
\end{theorem}

The group $\out(G)$ was initially studied by Rips and Sela \cite{RS94}. Shortly later, Sela \cite{Sel97} precisely described $\out(G)$ as a finite-index supergroup of a central extension of a direct product of mapping class groups of surfaces as follows. Also see \cite{Lev05} for a simpler proof using the canonical JSJ splitting of $G$.

\begin{theorem}\cite[Theorem 1.9]{Sel97}\label{IntroThm: CenSplitOfOUT(G)}
    Let $G$ be a torsion-free one-ended hyperbolic group. There is a central extension $$1\to \Z^n\to \out'(G)\to \prod_{i=1}^m\pmcg(\Sigma_i)\to 1$$ where $\out'(G)$ has finite index in $\out(G)$ and $\Sigma_i$ are finite-type surfaces.
\end{theorem}

As a direct corollary of Theorem \ref{IntroThm: Out(G)} and Theorem \ref{MainThm1}, we show that

\begin{corollary}\label{OutQT2}
    The central extension in Theorem \ref{IntroThm: CenSplitOfOUT(G)} has a bounded Euler class.
\end{corollary}

After this paper first appeared, Hadziosmanovic and Mangioni \cite{HM26} independently established Corollary \ref{OutQT2}. We both rely on Levitt's description of $\out(G)$ (see Theorem \ref{IntroThm: CenSplitOfOUT(G)}). While their proof tracks boundedness directly through this description, we deduce it from property QT. Moreover, they extend the result to one-ended hyperbolic groups with torsion, where the JSJ decomposition involves orbifolds and it has to be worked out that the relevant extension is central. However, our QT result is still not covered by their approach.

\subsection{Property PH}\label{subsec: PHintro}

We now describe property PH and its relationship with property QT in more detail. In \cite{BBF21}, Bestvina--Bromberg--Fujiwara showed that both residually finite hyperbolic groups and mapping class groups of finite-type surfaces (without boundary) have property QT. More examples of property QT include Coxeter groups \cite{DJ99} and most 3-manifold groups \cite{HNY25}. From the perspective of hierarchically hyperbolic group (HHG) theory, Hagen and Petyt \cite{HP22} provide a sufficient condition for an HHG to have property QT. Most of these results have been generalized to a class of relatively hierarchically hyperbolic groups (RHHG) by the first author in \cite{Tao26}.

In this paper, we introduce property PH as a generalization of property QT that retains its hyperbolic features while allowing focal factors. Yang and the second author \cite{WY25} proved that if a finitely generated group $G$ has property PH and each factor action satisfies the shadowing property (see \cite[\textsection 2.2]{WY25} for the definition), then $G$ has locally uniform exponential growth. Motivated by this result, we hope to find more interesting examples of property PH and study its relationship with property QT. A key difference from property QT is that a finitely generated group with property PH may contain distorted elements, while a finitely generated group with property QT cannot (cf. \cite[Lemma 2.5]{HNY25}). This distinction is essential for solvable Baumslag--Solitar groups and Sol $3$-manifold groups. In particular, we show that property QT is strictly stronger than property PH by proving:

\begin{proposition}\label{IntroProp: PHButNotQT}
    ~
    \begin{enumerate}
        \item Solvable Baumslag-Solitar groups $BS(1,n)=\langle a,t\mid tat^{-1}=a^n\rangle$, where $n\ge 2$, have property PH but do not have property QT.
        \item Fundamental groups of Anosov mapping tori have property PH but do not have property QT.
    \end{enumerate}
\end{proposition}

As we know, property QT is a commensurability invariant \cite{BBF21} but not a quasi-isometry invariant \cite{But25}. Following the same proof in \cite{BBF21} using induced actions, property PH is also a commensurability invariant, but it is unknown whether it is a quasi-isometry invariant (Question \ref{Que: PHQIInv}). Following \cite{Tao26}, a finitely generated group $G$ has \textit{property QT$_0$} if it has property QT from a diagonal action. Similarly, we say a group $G$ has \textit{property PH$_0$} if it has property PH from a diagonal action (see Definition \ref{Def: PH}).
By a result of Button \cite[Theorem 6.1]{But25}, a group with property QT or PH virtually has property QT$_0$ or PH$_0$, respectively. We show that property PH$_0$ and QT$_0$ are strictly stronger than property PH and QT respectively, and thus they are not commensurability invariants by proving the following.

\begin{proposition}\label{IntroProp: QTButNotPH0}
    The (3,3,3)-triangle group $T=\langle a, b, c\mid  a^2=b^2=c^2=(ab)^3=(bc)^3=(ac)^3=1\rangle$ has property QT but does not have property PH$_0$.
\end{proposition}

\begin{figure}
    \centering
\tikzset{every picture/.style={line width=0.75pt}} 

\begin{tikzpicture}[x=0.75pt,y=0.75pt,yscale=-0.8,xscale=0.8]

\draw  [color={rgb, 255:red, 0; green, 0; blue, 0 }  ,draw opacity=1 ][fill={rgb, 255:red, 155; green, 155; blue, 155 }  ,fill opacity=0.5 ] (133,146.5) .. controls (133,87.13) and (223.66,39) .. (335.5,39) .. controls (447.34,39) and (538,87.13) .. (538,146.5) .. controls (538,205.87) and (447.34,254) .. (335.5,254) .. controls (223.66,254) and (133,205.87) .. (133,146.5) -- cycle ;
\draw  [fill={rgb, 255:red, 80; green, 227; blue, 194 }  ,fill opacity=0.5 ] (154,147.5) .. controls (154,105.8) and (205.93,72) .. (270,72) .. controls (334.07,72) and (386,105.8) .. (386,147.5) .. controls (386,189.2) and (334.07,223) .. (270,223) .. controls (205.93,223) and (154,189.2) .. (154,147.5) -- cycle ;
\draw  [fill={rgb, 255:red, 184; green, 233; blue, 134 }  ,fill opacity=0.5 ] (288,146.5) .. controls (288,104.8) and (339.93,71) .. (404,71) .. controls (468.07,71) and (520,104.8) .. (520,146.5) .. controls (520,188.2) and (468.07,222) .. (404,222) .. controls (339.93,222) and (288,188.2) .. (288,146.5) -- cycle ;
\draw  [fill={rgb, 255:red, 248; green, 231; blue, 28 }  ,fill opacity=0.5 ] (304.75,146.5) .. controls (304.75,129.52) and (318.52,115.75) .. (335.5,115.75) .. controls (352.48,115.75) and (366.25,129.52) .. (366.25,146.5) .. controls (366.25,163.48) and (352.48,177.25) .. (335.5,177.25) .. controls (318.52,177.25) and (304.75,163.48) .. (304.75,146.5) -- cycle ;

\draw (321,51) node [anchor=north west][inner sep=0.75pt]   [align=left] {PH};
\draw (219,94) node [anchor=north west][inner sep=0.75pt]   [align=left] {PH$_0$};
\draw (412,88) node [anchor=north west][inner sep=0.75pt]   [align=left] {QT};
\draw (322,124) node [anchor=north west][inner sep=0.75pt]   [align=left] {QT$_0$};

\end{tikzpicture}
    \caption{The inclusion relationship between properties PH, PH$_0$, QT and QT$_0$.}
    \label{Fig: Inclusion}
\end{figure}
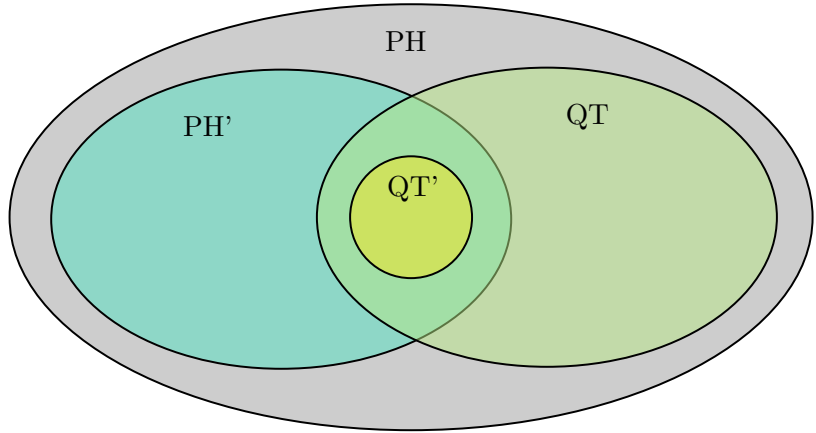

Figure \ref{Fig: Inclusion} illustrates the relations between properties PH, PH$_0$, QT and QT$_0$. To prove Proposition \ref{IntroProp: QTButNotPH0}, we study thoroughly the groups with these properties from lineal actions in Section \ref{sec: PHLineal}. 

Let $M$ be a connected, compact, orientable 3-manifold. As shown in \cite[Theorem 1.1]{HNY25}, the fundamental group $\pi_1(M)$ has property QT if and only if no summand in the sphere-disk decomposition of $M$ supports either Sol or Nil geometry. However, as mentioned in Proposition \ref{IntroProp: PHButNotQT}, fundamental groups of Anosov mapping tori have property PH. In fact, we prove the following. 

\begin{theorem}\label{IntroThm: 3Manifold}
    Let $M$ be a connected, compact, orientable 3-manifold. If no summand in the sphere-disk decomposition of $M$ supports Nil geometry, then $\pi_1(M)$ has property PH.
\end{theorem}

We remark that there are also many counterexamples of property PH or PH$_0$. In \cite{BFFG24}, a group $G$ is said to have \textit{property (NL)} if no isometric actions of $G$ on a hyperbolic space admits a loxodromic element. Moreover, if all of the finite-index subgroups of $G$ have property (NL), $G$ is said to have \textit{property hereditary (NL)}. The action of a group with property (NL) on any hyperbolic space is either elliptic or horocyclic. Thus, any infinite group with property (NL) does not have property PH$_0$. Moreover, any infinite group with property hereditary (NL) does not have property PH. We refer the readers to \cite{BFFG24} for examples of groups satisfying property (hereditary) (NL). After this paper first appeared, Balasubramanya, Niebler and Shapiro \cite{BNS26} characterized which Coxeter groups have property (NL). Since every finitely generated Coxeter group has property QT \cite{DJ99}, the existence of infinite Coxeter groups with property (NL) also implies that neither property (NL) nor PH$_0$ is a commensurability invariant. 

We also remark that there are some other literature to study isometric group actions on finite products (with $\ell^1$ or $\ell^2$ metric) of hyperbolic spaces. For example, Button \cite{But25} gave a criterion for a generalized Baumslag-Solitar group to virtually have a \textit{product acylindrical action}, i.e. an acylindrical diagonal action on a finite $\ell^1$-product of hyperbolic spaces. After that, Balasubramanya and Fernos \cite{BF25} systematically studied \textit{(AU)-acylindrical actions} on a finite $\ell^2$-product of hyperbolic spaces with general type factors to unify both notions of S-arithmetic lattices and acylindrically hyperbolic groups.

\subsection*{Structure of the paper}
The paper is organized as follows. Section \ref{sec: preliminary} is devoted to recalling some preliminary materials about Gromov-hyperbolic spaces, group actions, quasimorphisms, and central extensions. In Section \ref{sec: PHLineal}, we show that property QT implies PH and study the groups with either property from lineal actions. In Section \ref{sec: FirstExample}, we study some representative examples and prove Proposition \ref{IntroProp: PHButNotQT}, Proposition \ref{IntroProp: QTButNotPH0}, and Theorem \ref{IntroThm: 3Manifold}. In Section \ref{Sec: CentQH}, we provide several equivalent conditions for a central extension to have a bounded Euler class. As a result, we prove part of Theorem \ref{MainThm1} there. In Section \ref{Sec: PHUnderQuotient}, we construct hyperbolic spaces for central quotients and complete the proof of Theorem \ref{MainThm1}. This leads to applications to mapping class groups and outer automorphism groups of torsion-free one-ended hyperbolic groups. At the end, we collect some open questions about properties PH and QT in Section \ref{sec: OpenQues}.

\subsection*{Acknowledgments}
The authors are grateful to Wenyuan Yang for raising the question of how properties QT and PH behave under central extensions and for giving many helpful suggestions on the first draft of this paper. B. Tao was supported by JST SPRING, Grant Number JPMJSP2110. R. Wan was supported by NSFC No.12471065 \& 12326601 and in part by Science and Technology Commission of Shanghai Municipality (No. 22DZ2229014).

\section{Preliminary}\label{sec: preliminary}
\subsection{Gromov-hyperbolic spaces and group actions}\label{subsec: HypSpace}

Let $X,Y$ be metric spaces. Given constants $\lambda\ge 1$, $\epsilon\ge 0$, a map $f:X\to Y$ is called 

\begin{enumerate}
\item \textit{$\epsilon$-cobounded} if for any $y\in Y$, there exists $x\in X$ with $|f(x)-y|\le \epsilon$.
\item \textit{$(\lambda, \epsilon)$-quasi-isometric embedding} if for any $x,y\in X$, 
    \[
        \frac{1}{\lambda}|x-y|-\epsilon\le |f(x)-f(y)|\le \lambda|x-y|+\epsilon.
    \]
\item \textit{$(\lambda, \epsilon)$-quasi-isometry} if $f$ is a $(\lambda, \epsilon)$-quasi-isometric embedding and $\epsilon$-cobounded. When such a map exists, the two spaces $X$ and $Y$ are said to be \textit{quasi-isometric.} 
\end{enumerate}

Let $(X, |\cdot|)$ be a geodesic metric space. For any two points $x,y\in X$, denote $[x,y]$ as  a choice of a geodesic segment between $x$ and $y$. A \textit{geodesic triangle} $\Delta=\Delta(x,y,z)$ in $X$ consists of three points $x, y, z\in X$ and three geodesic segments $[x, y], [y, z]$ and $ [z, x]$.

A geodesic metric space $X$ is called \textit{(Gromov) $\delta$-hyperbolic} for a constant $\delta \geq 0$ if every geodesic triangle in $X$ is \textit{$\delta$-thin}: each of its sides is contained in the $\delta$-neighborhood of the union of the other two sides. A geodesic metric space is called \textit{hyperbolic} if it is $\delta$-hyperbolic for some $\delta\ge 0$.

A \textit{$(\lambda, \epsilon)$-quasi-geodesic} in a metric space $X$ is the image of a 
 $(\lambda, \epsilon)$-quasi-isometric embedding $c: I \rightarrow X$, where $I$ is an interval (possibly bounded or unbounded). For simplicity, a $(\lambda,\lambda)$-quasi-geodesic will be referred to as a \textit{$\lambda$-quasi-geodesic}.

Let $X$ be a  hyperbolic space. Two quasi-geodesic rays in $X$ are said to be \textit{asymptotic} if the Hausdorff distance between them is finite. Being asymptotic is an equivalence relation on quasi-geodesic rays. The \textit{Gromov boundary} of $X$, denoted by $\partial X$, is defined to be the set of equivalence classes of quasi-geodesic rays in $X$. 
Suppose that a group $G$ acts isometrically on a hyperbolic space $X$ with a basepoint $o$. The \textit{limit set} of $G$ is defined by $\Lambda G=\overline{Go}\cap \partial X$, and is independent of the choice of the basepoint.

By Gromov \cite{Gro87}, the isometries of a hyperbolic space $X$ can be divided into three classes. A nontrivial element $g\in \Isom(X)$ is called \textit{elliptic} if some $\langle g\rangle$-orbit is bounded. Otherwise, it is called \textit{loxodromic} (resp. \textit{parabolic}) if it has exactly two fixed points (resp. one fixed point), denoted by $g^+,g^-$, in the Gromov boundary $\partial X$ of $X$. 

An isometric group action of $G$ on a metric space $X$ is called \textit{metrically proper} if for any $x\in X$ and any $R>0$, the set $\{g\in G\mid |x-gx|\le R\}$ is finite. For brevity, we will write proper actions to mean metrically proper actions in this paper. A proper and cobounded group action will be referred to as a \textit{geometric} action.

Following \cite[Lemma 3.6]{ABO19}, we say two isometric actions $G\curvearrowright X$ and $G\curvearrowright Y$ are \textit{weakly equivalent} if there exists a coarsely $G$-equivariant quasi-isometry from a $G$-orbit in $X$ to a $G$-orbit in $Y$.

According to Gromov \cite{Gro87}, all isometric group actions on hyperbolic spaces can be classified into the following five types:
\begin{enumerate}
   \item \textit{elliptic} if the orbit is bounded;
   \item \textit{horocyclic} if it has no loxodromic element;
   \item \textit{lineal} if it has a loxodromic element and any two loxodromic elements have the same fixed points in Gromov boundary;
   \item \textit{focal} if it has a loxodromic element, is not lineal and any two loxodromic elements have one common fixed point;
   \item \textit{of general type} if it has two loxodromic elements with no common fixed point.
\end{enumerate}
Moreover, a lineal action is called \textit{orientable} if no group element permutes the two limit points of this action. 

A \textit{quasi-line} (resp. \textit{quasi-tree}) is a metric space quasi-isometric to $\R$ (resp. a tree) equipped with the path metric. In this paper, we assume each quasi-line $L$ to be oriented by choosing an ordering of its limit points $L^+,L^-$. It is easy to see that the orbit of a lineal action is a quasi-line. 

\begin{lemma}\cite[Corollary 3.6]{Man06}
    Any isometry on a quasi-tree is either elliptic or loxodromic.
\end{lemma}

\subsection{Quasimorphisms and lineal actions}

Let $G$ be a group. A function $\phi:G\to \R$ is called a \emph{quasimorphism} if its \emph{defect} 
        \[D(\phi):=\sup_{g,h\in G} |\phi(g)+\phi(h)-\phi(gh)|\] 
        is finite. 
    A quasimorphism $\phi$ on $G$ is \emph{homogeneous} if $\phi(g^n)=n\phi(g)$ for any $n\in \mathbb Z$ and any $g\in G$.

    \begin{remark}\label{Rmk: HomoVersion}
        It is well-known \cite[Lemma 2.21]{Cal09} that the homogenization of a quasimorphism $\phi$ defined by
        $$\varphi(g)=\lim_{n\rightarrow \infty}\frac{\phi(g^n)}{n}$$
        is a homogeneous quasimorphism, so that $\|\phi-\varphi\|_\infty$ is bounded above by $D (\phi)$. Moreover, we have an estimate $D (\varphi)\le 2D (\phi)$. 
        A homogeneous quasimorphism $\varphi$ is a class function: $\varphi(g)=\varphi(hgh^{-1})$ for any $h, g\in G$. This implies that $|\varphi([g,h])|\le D (\varphi)$ for any $g, h\in G$. 
    \end{remark}

    The following notion is a natural generalization of quasimorphisms. 
    \begin{defn}\cite{FK16}\label{Def: QH}
        Let $G$ be a group and $H$ be a group equipped with a proper left-invariant metric. Let $\phi: G \to H$ be a set-theoretic map. Define $\Delta(\phi) \subseteq H$, the \emph{defect set} of $\phi$, by $$\Delta(\phi)=\{\phi(h)^{-1}\phi(g)^{-1}\phi(gh): g,h\in G\}.$$
        The map $\phi: G\to H $ is called a \emph{quasi-homomorphism} if its defect set $\Delta(\phi)$ is finite. 
    \end{defn}

    \begin{remark}\label{Qhom}
        One may also define the defect set of $\phi$ by $\bar \Delta(\phi)=\{\phi(g)\phi(h)\phi(gh)^{-1}: g,h\in G\}.$ This change does not affect the definition of quasi-homomorphisms (see \cite[Proposition 2.3]{Heu20}).
    \end{remark}

    A useful fact about extending quasi-homomorphisms from a finite-index subgroup is the following result, which is clear by the proof of \cite[Lemma 3.1]{Ish14}.
    \begin{lemma}\label{Lem: TransferMap}
        Let $H$ be a finite-index normal subgroup of $G$. Then any quasi-homomorphism $\phi: H\to \R^n$ induces a quasi-homomorphism $\hat \T(\phi): G\to \R^n$ whose restriction on $H\cap Z(G)$ is equal to $\phi$.
    \end{lemma}


We introduce the following notion for an orientable lineal action. 

\begin{defn}[Oriented quasimorphism]\label{Def: TranslationNumber}
    Let $G$ be a group that acts orientably and lineally on a quasi-line $L$. An \emph{$L$-oriented} quasimorphism is a nonzero homogeneous quasimorphism $\varphi:G\to \R$ such that
    \begin{itemize}
        \item $\varphi(g)>0$ if $g$ is loxodromic with $g^+=L^+$ and $g^-=L^-$;
        \item $\varphi(g)<0$ if $g$ is loxodromic with $g^+=L^-$ and $g^-=L^+$;
        \item $\varphi(g)=0$ if $g$ is elliptic.
    \end{itemize}
\end{defn}

The following lemma is known to experts and serves as a bridge between lineal actions and quasimorphisms.

\begin{lemma}\label{Lem: TFAE}
    Let $G$ be a group. For any orientable and lineal $G$-action on a quasi-line $L$, there exists at least one $L$-oriented quasimorphism associated to this action.

    Conversely, any nonzero homogeneous quasimorphism $\varphi: G\to \R$ can be realized as an $L$-oriented quasimorphism for some quasi-line $L$.
\end{lemma}
\begin{proof}
    For the forward direction, the construction of a canonical $L$-oriented quasimorphism (called the \emph{Busemann quasimorphism} or \emph{Busemann pseudocharacter}) is given in \cite[\S 4]{Man08} (see also \cite[Proposition 3.7]{CCMT15}). 
    The converse was previously known to Manning \cite[Proposition 3.1, Proposition 5.5]{Man06} and also appears as \cite[Lemma 4.15]{ABO19}.
\end{proof}

\begin{remark}\label{Rmk: STL}
    The \textit{stable translation length} of an isometry $g$ on $L$ is defined by $$\tau(g):=\lim_{n\to \infty}\frac{d(x,g^nx)}{n}$$ for any basepoint $x\in L$. It is not hard to see from the construction of Busemann quasimorphisms (cf. \cite[\S 4]{Man08}, \cite[Proposition 3.7]{CCMT15}) that $|\beta(g)|=\tau(g)$, where $\beta$ is the Busemann quasimorphism.
\end{remark}

\begin{lemma}\label{Lem: UniformBounded}
Suppose that a group $G$ has an orientable lineal action on a hyperbolic space $X$ with a basepoint $o\in X$. Then there exists $C>0$ such that for any $g,h\in G$, $|o-[g,h]o|\leq C$.
\end{lemma}
\begin{proof}
    When restricted to the orbit of the $G$-action in $X$, we know that $G$ has an orientable lineal action on a quasi-line. According to Lemma \ref{Lem: TFAE} and the above remark, there exists a homogeneous quasimorphism $\beta: G\to \R$ such that $|\beta(g)|$ measures the stable translation length of $g$ in $X$. As shown in Remark \ref{Rmk: HomoVersion}, $|\beta([g,h])|\le D(\beta)$ for any $g,h\in G$.

    If $[g,h]$ is a loxodromic element in $X$, then the conclusion follows from \cite[Proposition 10.6.4]{CDP06}. If $[g,h]$ is an elliptic element in $X$, then the conclusion follows from \cite[Corollary 6.7]{Osi16} and the cobounded action $G\curvearrowright G\cdot o$.
\end{proof}

\begin{remark}
    We caution the readers that the conclusion of Lemma \ref{Lem: UniformBounded} depends heavily on the orientability. For example, let $G=\langle r, s\mid s^2=(rs)^2=1\rangle$ be the infinite dihedral group and $S=\{r,s\}$. Then the $G$-action on the Cayley graph $\G(G,S)$ is non-orientable lineal. It is direct to compute that $|[r^n,s]|_S=|r^{2n}|_S\to \infty$ as $n\to \infty$.
\end{remark}

\subsection{Central extensions}\label{subsec: CentralExtension}
Consider a short exact sequence of groups
\begin{equation}\label{Ext}
    1\to Z\xrightarrow{\iota}E\xrightarrow{\pi}G\to 1.
\end{equation}

\begin{defn}
    A map $s:G\to E$, not necessarily a homomorphism, is called a \textit{(set-theoretical) section} if $\pi\circ s= Id_G$. A section $s$ is called \textit{normalized} if $s(1)=1$. A normalized section $s$ is called a \textit{quasi-splitting} if $s$ is a quasi-homomorphism (see Definition \ref{Def: QH}). Following \cite{FK16}, we say that (\ref{Ext}) \textit{quasi-splits} if it admits a quasi-splitting. We also say that (\ref{Ext}) \textit{virtually quasi-splits} if there exists a finite-index subgroup $E'\le E$ such that the short exact sequence 
    $$1\to Z'\xrightarrow{\iota|_{Z'}} E'\xrightarrow{\pi|_{E'}}G'\to 1$$
    quasi-splits, where $Z'=Z\cap E'$ and $G'=\pi(E')$. Following \cite{Wan25}, a map $\phi:E\to Z$ is called a \textit{quasi-retraction}\footnote{This definition is different from that in \cite{Alo94}, where a quasi-retraction is defined metrically. } if $\phi$ is a quasi-homomorphism and $\phi|_Z=\mathrm{id}_Z$. 
\end{defn}

When $Z$ is a central subgroup of $E$, (\ref{Ext}) is called a \textit{central extension} of $G$. In this case, we also say that $E$ is a central extension of $G$ by $Z$ and that $G$ is a central quotient of $E$ by $Z$. In what follows, we assume $Z$ to be finitely generated.
Let $s: G \to E$ be a section. For $g_1, g_2 \in G$, the element
$s(g_1)s(g_2)s(g_1g_2)^{-1}$ lies in the kernel of $\pi$, hence in the image of $\iota$. Up to
identifying $\iota(Z)$ with $Z$ we may thus define the cochain $\omega_s \in C^2(G, Z)$ given by
$$\omega_s(g_1, g_2) = s(g_1)s(g_2)s(g_1g_2)^{-1}.$$ We say $\omega_s$ is \textit{bounded} if the set $\{\omega_s(g_1,g_2): g_1,g_2\in G\}$ is finite in $Z$.

Let us recall the following well-known facts (see e.g. \cite[Chapter 4]{Bro82}):
\begin{itemize}
    \item[(1)]  The cochain $\omega_s$ is a cocycle;
    \item[(2)]  If $s'$ is another section of $\pi$, then $\omega_{s'}$ is cobordant to $\omega_s$; therefore, the class $[\omega_s] \in H^2(G, Z)$ does not depend on the choice of $s$, and will be called the \textit{Euler class} of the extension;
    \item[(3)]  If $\omega'$ is any representative of the Euler class, there exists a section $s': G \to E$ such that $\omega'=\omega_{s'}$;
    \item[(4)]  The cocycle $\omega_s$ is normalized, in the sense that $\omega_s(1,g)=1=\omega_s(g,1)$, if and only if $s$ is normalized.
\end{itemize}
The Euler class of a central extension is called \textit{bounded} if it can be represented by a bounded cocycle.

\section{Property PH and lineal actions}\label{sec: PHLineal}


\begin{defn}\label{Def: PH}
    A group $G$ has \textit{property PH} if there exist finitely many hyperbolic spaces $X_1,\cdots,X_n$, such that $G$ acts properly on $\prod_{i=1}^nX_i$ with $\ell^1$-metric and the stabilizer $\stab_G(X_i)$ acts coboundedly on $X_i$ for each $i$. To specify the actions, we also say $G$ has \textit{property PH from actions on $X_1,\dots,X_n$} in this case. 
    
    If the above action $G\curvearrowright\prod_{i=1}^nX_i$ can be diagonal, then we say $G$ has \textit{property PH$_0$ (from actions on $X_1,\dots,X_n$)}. 
\end{defn}

In this paper, we use $\G(G,S)$ to denote the Cayley graph of a group $G$ with respect to a generating set $S$. To begin with, we show that property QT always implies property PH. 


\begin{lemma}\label{Lem: QTImpPH}
     Let $G$ be a finitely generated group. If $G$ has property QT$_0$ or QT, then $G$ has property PH$_0$ or PH, respectively.
\end{lemma}
\begin{proof}
    By \cite[Theorem 6.1]{But25}, a finitely generated group with property QT virtually has property QT$_0$. It suffices for us to show that property QT$_0$ implies property PH$_0$. Let $G$ be a finitely generated group with property QT$_0$. By definition, there exist finitely many quasi-trees $T_i(1\le i\le n)$ on which $G$ acts isometrically so that for any finite generating set $F$ of $G$ and $o\in X=\prod_{i=1}^nT_i$, the orbit map $\G(G,F)\to X, g\mapsto go$ is a quasi-isometric embedding. This quasi-isometric embedding of the orbit map implies that the diagonal action $G\curvearrowright X=\prod_{i=1}^nT_i$ is proper. By \cite[Remark 3.2]{Man06}, any isometric group action of a finitely generated group on a quasi-tree can be replaced by a quasi-tree on which the action is cobounded and weakly equivalent to the original action. This shows that $G$ has property PH$_0$. 
\end{proof}

\begin{remark}
    Button's result \cite[Theorem 6.1]{But25} is only stated for finite products of connected graphs. However, for any hyperbolic space $X$, we can construct a connected graph $Y$, called a \emph{Rips graph}, whose vertex set is $X$ and whose edges connect two points $x,y\in X$ whenever $|x-y|\le 1$. The graph $Y$ is quasi-isometric to $X$, and any action on $X$ passes to a weakly equivalent action on $Y$. Thus, Button's result also works for finite products of hyperbolic spaces. 
\end{remark}

If each $X_i$ in Definition \ref{Def: PH} is a quasi-line, then we say $G$ has \textit{property PH (or property PH$_0$ respectively) from lineal actions}. Since horocyclic actions cannot be cobounded and elliptic actions can be removed in a proper diagonal action, the simplest objects for us to study are those groups with property PH from lineal actions. 
For the rest of this section, we are going to describe the class of groups with property PH or PH$_0$ from lineal actions. 

In this paper, we say a virtually abelian group $G$ has \textit{rank} $r$ if $G$ contains a finite-index subgroup which is isomorphic to $\Z^r$. A group $G$ is called \textit{FC} if every conjugacy class of $G$ contains only finitely many elements. Note that a group with finitely many commutators is FC. A group $G$ is called \textit{FZ} if its center $Z(G)$ has finite-index in $G$. 


\begin{lemma}\label{Lem: OriLineal}
    Assume a group $G$ has property PH$_0$ from orientable lineal actions $G\curvearrowright L_1,\cdots, G\curvearrowright L_m$. Then $G$ is a finitely generated FZ group of rank $r\le m$. Moreover, there exists a subset $\{L_1',\cdots,L_r'\}\subseteq \{L_1,\cdots,L_m\}$ so that the diagonal action $G\curvearrowright \prod_{i=1}^rL_i'$ is geometric. 
\end{lemma}

\begin{proof}
    By Lemma \ref{Lem: TFAE}, there exists an $L_i$-oriented quasimorphism $\varphi_i$ on $G$ for each $i=1,\dots,m$. Properness of the diagonal action $G\curvearrowright \prod_{i=1}^mL_i$ is equivalent to say that the set
    \[T_k:=\{g\in G\mid |\varphi_i(g)|\le k,\ \ \forall 1\le i\le m\}\]
    is finite for any $k\ge 0$.

    If there exist constants $c_1,\dots,c_{m-1}\in \R$ such that $\varphi_m=\sum_{i=1}^{m-1} c_i\varphi_i$, then for any $g\in G$ with $|\varphi_i(g)|\le k$ for $1\le i\le m-1$, we have $|\varphi_m(g)|\le k\cdot\max\{|c_i|\mid 1\le i\le m-1\}$. This implies that the set \[\{g\in G\mid |\varphi_i(g)|\le k,\ \ \forall 1\le i\le m-1\}\] is finite. Equivalently, the diagonal action $G\curvearrowright \prod_{i=1}^{m-1}L_i$ is still proper. By choosing $\{\varphi_i\mid 1\le i\le r\}\subseteq \{\varphi_i\mid 1\le i\le m\}$ to be a maximal linearly independent subset, the above arguments show that the diagonal action $G\curvearrowright \prod_{i=1}^{r}L_i$ is still proper.

    Since $T_k$ is finite for any $k\ge 0$ and $G=\bigcup_{k=1}^{\infty} T_k$, we can arrange the group elements of $G$ in a sequence $(g_n)_{n\ge 1}$ such that $T_k$ is exactly the set of the first $n_k$ elements, where $n_k=|T_k|\in \N$ is non-decreasing with respect to $k$. 
    Define a map $\Phi: G \to \R^r$ by $\Phi(g)=(\varphi_1(g),\cdots, \varphi_r(g))^T$ for any $g\in G$. For any $k\in \N$, define an ($r, n_k$)-matrix $M_k$ by $M_k=(\Phi(g_1),\dots,\Phi(g_{n_k}))$. Note that for any $k<k'$, $M_k$ is contained in $M_{k'}$ as the first $n_k$ columns. Therefore, $\rk(M_k)$ is non-decreasing with respect to $k$. Moreover, $\rk(M_k)\le r$ by definition. Thus, there exists $r'\le r$ and $k_0\in\N$ such that $\rk(M_k)=r'$ for any $k\ge k_0$. 

    \begin{claim}
        $r'=r$.
    \end{claim}
    \begin{proof}[Proof of Claim]
        Suppose that $r'<r$. After passing to a subsequence $(M_{k_l})_{l\in\N}$ and reordering the quasimorphisms if necessary, we can assume that the first $r'$ row vectors of $M_{k_l}$ are linearly independent for any $l\in \N$. Thus, there exist constants $c_i\in\R$ ($i=1,\dots, r'$) such that $\varphi_r(g_n)=\sum_{i=1}^{r'}c_i\varphi_i(g_n)$ for any $n\ge 1$. This contradicts with the linear independency of $\varphi_i$ ($i=1,\dots,r$) on $G$. 
    \end{proof}

    Since $\rk(M_{k_0})=r$, we can choose $r$ column vectors of $M_{k_0}$ that are linearly independent. WLOG, let $\Phi(g_1),\dots,\Phi(g_r)$ be linearly independent. For any $g\in G$, there exist $b_n\in\R$ ($n=1,\dots, r$) such that 
    \[\Phi(g) = \sum_{n=1}^{r}b_n\Phi(g_n).\]
    For any number $x\in \R$, we use $\lfloor x\rfloor$ to denote the floor function. By definition of homogeneous quasimorphisms, 
    \[\left |\varphi_i \left (g\prod_{n=1}^{r}g_n^{-\lfloor b_n\rfloor} \right )-\varphi_i(g)+\sum_{n=1}^{r}\lfloor b_n\rfloor\varphi_i(g_n)\right |\le rD(\varphi_i),\]
    for any $1\le i\le r$. Thus,
    \begin{align}\label{cobdd}
        \left |\varphi_i \left (g\prod_{n=1}^{r}g_n^{-\lfloor b_n\rfloor} \right ) \right |\le \sum_{n=1}^{r} |\varphi_i(g_n)|+rD(\varphi_i)\le r(k_0+D(\varphi_i)),
    \end{align}
    for any $1\le i\le r$. Since $\varphi_j$ is a linear combination of $\varphi_1,\dots,\varphi_r$ for $j>r$, $\left |\varphi_i \left (g\prod_{n=1}^{r}g_n^{-\lfloor b_n\rfloor} \right ) \right |$ is also bounded above by different constants. This shows that there exists $k\ge 0$ independent of $g$ such that $g\in \langle T_k\rangle$. In particular, $G$ is finitely generated.
    
    Lemma \ref{Lem: UniformBounded} and properness of $G\curvearrowright\prod_{i=1}^rL_i$ imply that $G$ is an FC group. Since a finitely generated FC group is FZ (see \cite[\S 2]{Neu51}), $G$ is FZ. Moreover, (\ref{cobdd}) also shows that the diagonal action $G\curvearrowright\prod_{i=1}^rL_i$ is cobounded, and thus geometric. Therefore, the rank of $G$ equals $r$. 
\end{proof}

\begin{remark}
    Conversely, any finitely generated FZ group $G$ has PH$_0$ from orientable lineal actions \cite[Proposition 5.26]{Wan25}.
\end{remark}

Now we can give a description of groups that have property PH$_0$ from lineal actions.

\begin{proposition}\label{Prop: AllQausiLines}
    Assume a group $G$ has property PH$_0$ from lineal actions $G\curvearrowright L_1,\cdots,G\curvearrowright L_m$. Then $G$ fits into a short exact sequence $$1\to H\to G\to \Z_2^n\to 1,$$ where $n\le m$ and $H$ is a finitely generated FZ group of rank $r\le m$. Moreover, there exists a subset $\{L_1',\cdots,L_r'\}\subseteq \{L_1,\cdots,L_m\}$ so that the diagonal action $G\curvearrowright \prod_{i=1}^rL_i'$ is geometric. 
\end{proposition}

\begin{proof}
    For the forward direction, consider the induced action of $G$ on $\{\partial L_i: 1\le i\le m\}$. This gives a homomorphism $\rho: G\to \Z_2^m$, whose image is a finite group $\Z_2^n$ for some $n\le m$. The subgroup $H:=\ker \rho$ acts orientably on each $L_i$. By applying Lemma \ref{Lem: OriLineal} to $H$, we have that $H$ is a finitely generated FZ group of rank $r\le m$ and there exists a subset $\{L_1',\cdots,L_r'\}\subseteq \{L_1,\cdots,L_m\}$ so that $H\curvearrowright \prod_{i=1}^rL_i'$ is geometric. Since $H$ is of finite index in $G$, the diagonal action $G\curvearrowright \prod_{i=1}^rL_i'$ is also geometric. 
\end{proof}



\begin{corollary}\label{Cor: LinealPH}
    For any group $G$, the following are equivalent:
    \begin{enumerate}
        \item $G$ has property PH from lineal actions;
        \item $G$ is finitely generated and has property QT from lineal actions;
        \item $G$ is finitely generated and virtually abelian.
    \end{enumerate}
\end{corollary}
\begin{proof}
    ``(3)$\Rightarrow$(2)'' is clear since property QT is a commensurability invariant. ``(2)$\Rightarrow$(1)'' follows from Lemma \ref{Lem: QTImpPH}. ``(1)$\Rightarrow$(3)'' follows from Proposition \ref{Prop: AllQausiLines} and the fact that any group containing a finitely generated finite-index subgroup is also finitely generated.
\end{proof}

\begin{corollary}\label{Cor: NilNoPH}
    For any finitely generated group $G$ with sub-exponential growth, the following are equivalent:
    \begin{enumerate}
        \item $G$ has property PH;
        \item $G$ has property QT;
        \item $G$ is virtually abelian.
    \end{enumerate}
\end{corollary}

\begin{proof}
    ``(3)$\Rightarrow$(2)'' is clear. ``(2)$\Rightarrow$(1)'' follows from Lemma \ref{Lem: QTImpPH}. We only need to show ``(1)$\Rightarrow$(3)''. Let $G$ be a finitely generated group with sub-exponential growth that has property PH. By passing to a finite-index subgroup, we can assume that $G$ has property PH$_0$. 
    By definition, there exist finitely many unbounded hyperbolic spaces $X_1,\ldots, X_n$ on which $G$ acts coboundedly such that the diagonal action $G\curvearrowright X=\prod_{i=1}^nX_i$ is proper. Since any group admitting a focal or general type action on a hyperbolic space contains a non-abelian free sub-semigroup and thus has exponential growth, $G$ must act lineally on each $X_i$ for $1\le i\le n$. This implies that $G$ has property PH$_0$ from lineal actions. Then the conclusion follows from Corollary \ref{Cor: LinealPH}. 
\end{proof}

\section{First examples}\label{sec: FirstExample}

In this section, we are going to give some examples of groups with or without property PH. A group is called \textit{non-elementary} if it is not virtually abelian. 

\subsection{The Heisenberg group}\label{subsec: Heisenberg}
For every group $G$, Abbott-Balasubramanya-Osin \cite{ABO19} defined the set of \textit{hyperbolic structures} on $G$, denoted $\H(G)$, which consists of equivalence classes of
(possibly infinite) generating sets of $G$ such that the corresponding Cayley graph
is hyperbolic; two generating sets of $G$ are equivalent if the corresponding word
metrics on $G$ are bi-Lipschitz equivalent. These equivalence classes can be ordered
according to the amount of information the associated Cayley graph retains about
the group (see \cite[Section 1]{ABO19} for details); this makes $\H(G)$ a poset.

\begin{remark}
    With the language of hyperbolic structures, an equivalent definition for a group $G$ to have property PH$_0$ is that there exist finitely many hyperbolic structures $[S_1],\cdots,[S_n]\in \H(G)$ such that the diagonal action of $G$ on the $\ell^1$-product $\prod_{i=1}^n\G(G,S_i)$ is proper.
\end{remark}

The type of a hyperbolic structure $[S]\in \H(G)$ coincides with the type of the natural action $G\curvearrowright \G(G,S)$. In \cite{ABO19}, the sets of elliptic, lineal, horocyclic and general type hyperbolic structures on $G$ are denoted by $\H_e(G), \H_{\ell}(G), \H_{qp}(G)$ and $\H_{gt}(G)$, respectively.

Let $H=\langle a,b,c\mid [a,c]=[b,c]=1, [a,b]=c\rangle$ be the Heisenberg group. \cite[Example 4.23]{ABO19} shows that $\H(\Z^n)=\H_{e}(\Z^n)\sqcup \H_{\ell}^+(\Z^n)$ for $n\ge 2$ where $\H_{\ell}^+(\Z^n)$ is an antichain of cardinality continuum. The first result here is to show that $H$ and $\Z^n$ have isomorphic hyperbolic structures.
\begin{lemma}\label{Lem: HypStrOfHeisenberg}
    $\H(H)=\H_{e}(H)\sqcup \H_{\ell}^+(H)$ where $\H_{\ell}^+(H)$ is an antichain of cardinality continuum.
\end{lemma}
\begin{proof}
    Since $H$ has polynomial growth, it is clear that $\H_{qp}(H)=\H_{gt}(H)=\emptyset$ and thus $\H(H)=\H_{e}(H)\sqcup \H_{\ell}(H)$. Note that  there are two natural surjective homomorphisms $\phi_a: H\to \langle a\rangle, a\mapsto a, b\mapsto 1,c\mapsto 1$ and $\phi_b: H\to \langle b\rangle, a\mapsto 1, b\mapsto b, c\mapsto 1$. By the Svarc-Milnor lemma, these two homomorphisms give two non-equivalent generating sets in $\H_{\ell}^+(H)$. Thus $|\H_{\ell}^+(H)|\ge 2$. Since \cite[Theorem 4.22 (b)]{ABO19} shows that for every group $G$, the cardinality of $\H_{\ell}^+(G)$ is 0, 1 or at least continuum. Hence, the cardinality of $\H_{\ell}^+(H)$ is continuum. It remains to show that $\H_{\ell}(H)=\H_{\ell}^+(H)$.

    Pick an arbitrary element $[S]\in \H_{\ell}(H)$ and an arbitrary element $g\in H$ such that $g$ is loxodromic on $\G(H,S)$. Note that $\partial \G(H,S)=\{g^+,g^-\}$.  Since $c\in Z(H)$, one has that $cg=gc$ and thus $c$ fixes pointwise $\{g^+,g^-\}$. Moreover, it follows from another relation $ab=bac$ that $g$ can be written as $a^mb^nc^k$ for some $m,n,k\in \Z$ and $$aga^{-1}=a^{m+1}b^na^{-1}c^k=a^mb^nc^{k+n}=gc^n, \quad bgb^{-1}=ba^{m}b^{n-1}c^k=a^mb^nc^{k-m}=gc^{-m}.$$ 
    Since both $g$ and $c$ fix pointwise $\{g^+,g^-\}$, the above equations show that $a\cdot g^{\pm}=g^{\pm}, b\cdot g^{\pm}=g^{\pm}$. Since $\{a,b,c\}$ generates $H$, one gets that any element in $H$ fixes pointwise $\partial\G(H,S)$. This means that $H\curvearrowright \G(H,S)$ is orientable. Since $[S]$ is arbitrary,  we get that $\H_{\ell}(H)=\H_{\ell}^+(H)$.
\end{proof}

Since $H$ is a non-elementary group with polynomial growth, Corollary \ref{Cor: NilNoPH} shows that $H$ does not have property PH. Therefore, having property PH or not gives a distinction between $\Z^n(n\ge 2)$ and $H$. Another way to prove that $H$ does not have property PH will be given by Lemma \ref{Lem: CenComImpNoPH}.

It is well-known that $H$ fits into a central extension $$1\to \Z\to H\to \Z^2\to 1.$$ As a corollary of Theorem \ref{MainThm1}, this central extension has an unbounded Euler class. We will discuss central extensions and their Euler classes more detailedly in the next section.

\subsection{Crystallographic groups}
In this subsection, we give an equivalent condition for a crystallographic group to have property PH$_0$. As a corollary, we show that the (3,3,3)-triangle group has property QT but no PH$_0$. 

Recall that a \textit{crystallographic group} is a discrete subgroup of $\Isom(\R^r)$ that acts properly cocompactly on $\R^r$. By the first Bieberbach Theorem (see \cite{Szc25}), any such group $\Gamma$ fits into a short exact sequence of the form 
\begin{equation}\label{BieberThm}
    1\to A\to \Gamma\to F\to 1
\end{equation}
where $A$, called the \textit{translation subgroup}, is isomorphic to $\Z^r$ and $F$, the \textit{point group}, is a finite subgroup of the orthogonal group $O_r(\R)$. The above short exact sequence right-splits. This amounts to saying that $\Gamma$ can be written as a semi-direct product $A\rtimes_{\phi} F$ where $\phi$ is a homomorphism from $F$ to $\aut(\Z^r)=\gl_r(\Z)$.

\begin{lemma}\label{Lem: EquiCondForCryGp}
    Let $\Gamma\le \Isom(\R^r)$ be a crystallographic group. $\Gamma$ has property PH$_0$ if and only if the corresponding homomorphism $\phi: F\to \gl_r(\Z)$ factors through $\prod_{i=1}^r\aut(\Z)=\Z_2^r$.
\end{lemma}
\begin{proof}
    ``$\Rightarrow$'': Since $\Gamma$ is a virtually abelian group of rank $r$ and has property PH$_0$,  Proposition \ref{Prop: AllQausiLines} shows $\Gamma$ admits a geometric diagonal action on a finite product of quasi-lines $\prod_{i=1}^rL_i$. Since $[\Gamma: A]<\infty$, the diagonal action $A\curvearrowright \prod_{i=1}^rL_i$ is also geometric. Moreover, the induced action of $\Gamma$ on $\{\partial L_i: 1\le i\le r\}$ gives a homomorphism $\rho: \Gamma\to \Z_2^r$. Since $A$ is a torsion-free abelian group of rank $r$,  the geometric action $A\curvearrowright \prod_{i=1}^rL_i$ implies that no generating element of $A$ can swap the two points of $\partial L_i(i=1,\cdots,r)$. Hence, $A \le \ker \rho$ and this induces a quotient homomorphism $\bar \rho: F=\Gamma/A\to \Z_2^r$, which completes the proof.

    ``$\Leftarrow$'': Since $\Gamma\le \Isom(\R^r)$ is a crystallographic group, the translation subgroup $A$ is isomorphic to a lattice $\Z^r$ in $\R^r$ and $A$ acts by translation on $\Z^r$. When the corresponding homomorphism $\phi: F\to \gl_r(\Z)$ factors through $\prod_{i=1}^r\aut(\Z)=\Z_2^r$, the point group $F$ preserves each factor of $\Z^r$ and acts by permutation on the boundary of each factor. Since $\Gamma=A\rtimes_{\phi}F$, one gets an isometric action of $\Gamma$ on $\Z^r$ which preserves each $\Z$-coordinate. Since $[\Gamma: A]<\infty$, the action $\Gamma\curvearrowright \Z^r$ is also geometric, which gives the PH$_0$ property of $\Gamma$.
\end{proof}

Let $T=\langle a, b, c\mid  a^2=b^2=c^2=(ab)^3=(bc)^3=(ac)^3=1\rangle$ be the (3,3,3)-triangle group which acts properly and cocompactly on $\R^2$. Clearly $T$ is a crystallographic group with translation subgroup isomorphic to $\Z^2$. Recall that two groups are commensurable if they contain isomorphic finite-index subgroups. Petyt-Spriano \cite[Corollary 4.5]{PS23} used this triangle group to show that being an HHG is not a commensurability invariant.

\begin{proposition}\label{Prop: QTButNotPH0}
    The (3,3,3)-triangle group $T$ has property QT but does not have property PH$_0$. In particular, neither property PH$_0$ nor QT$_0$ is a commensurability invariant.
\end{proposition}
\begin{proof}
    From the short exact sequence (\ref{BieberThm}), we see that $T$ is virtually $\Z^2$. Since property QT is a commensurability invariant (see \cite{BBF21}), $T$ has property QT and thus property PH by Lemma \ref{Lem: QTImpPH}. Since the translation subgroup is torsion-free, any order-$3$ element of $T$ maps nontrivially to the point group. Thus the point group contains an element of order $3$, contradicting Lemma \ref{Lem: EquiCondForCryGp}. This shows that $T$ does not have property PH$_0$. 
\end{proof}

\subsection{Baumslag-Solitar groups}\label{subsec: BSGroup}
Recall that a Baumslag-Solitar group is defined by $BS(m,n)=\langle a,t\mid ta^mt^{-1}=a^n\rangle$. In this subsection, we prove that any solvable Baumslag-Solitar group $BS(1,n)$ has property PH$_0$ but does not have an HHG structure or property QT for $n\ge 2$. 

The reason why $BS(1,n)$ has neither an HHG structure nor property QT is due to the distortion of some infinite-order element in $BS(1,n)$; see \cite[Theorem 7.1]{DHS17} (of which the proof is corrected in \cite{DHS20}) and \cite[Lemma 2.5]{HNY25}.
\begin{lemma}\cite[Corollary 2.7]{HNY25}\label{Lem: BSNoQT}
    The Baumslag-Solitar group $BS(1, n)$ for $n \ge 2$ does not have property QT.
\end{lemma}

From now on, let us focus on $G_n = BS(1,n)=\langle a,t\mid tat^{-1}=a^n\rangle$ for $n\ge 2$. 

It is well-known and easy to prove that $G_n$ is isomorphic to the subgroup of $\sl_2(\R)$ generated by 
$$t=\left(
  \begin{array}{cc}
    \sqrt{n} & 0 \\
    0 & 1/\sqrt{n} \\
  \end{array}
\right) \text{ and } a=\left(
  \begin{array}{cc}
    1 & 1 \\
    0 & 1 \\
  \end{array}
\right).$$
Thus we obtain an action of $G_n$ on $\RH^2$ that factors through the action of $\sl_2(\R)$. It is easy to see that this action is focal and cobounded. Indeed, the generator $t$ is a loxodromic element and the other generator $a$ is a parabolic element and both generators have a common fixed point $\infty$ in $\partial \mathbb H^2$.
Moreover, a bounded fundamental domain of this action is given by the ``rectangle'' with vertices $\i,n\i,1+\i,1+n\i$. 

Another $G_n$-action on a hyperbolic space comes from the HNN-extension structure of $G_n$. Let $T_n$ be the corresponding Bass-Serre tree with a path metric $d_T$. Then $G_n \curvearrowright T_n$ is also focal with the unique fixed point $t^-\in \partial T_n$ and cobounded.

It is not hard to show that these two focal actions of $G_n$ are not weakly equivalent. Indeed there are no non-trivial elliptic elements with respect to the first action, while all elements from $[G_n, G_n]$ are elliptic with respect to the second action. As shown in \cite[Proposition 5.2(4)]{FSY04}, the diagonal action $G_n\curvearrowright \mathbb H^2\times T_n$ is proper. Thus, we obtain

\begin{proposition}\label{Prop: BSGroup}
    The solvable Baumslag-Solitar group $G_n=BS(1,n)$ has property PH$_0$ from focal actions on $\mathbb H^2$ and $T_n$.
\end{proposition}

As a direct corollary of Lemma \ref{Lem: BSNoQT} and Proposition \ref{Prop: BSGroup}, one gets that the class of finitely generated groups with property QT (resp. QT$_0$) is strictly contained in the class of groups with property PH (resp. PH$_0$). This is a complement to Lemma \ref{Lem: QTImpPH}.

\subsection{Fundamental groups of 3-manifolds}\label{subsec: 3Manifold}

Fix an element $\varphi\in \sl_2(\Z)$ that is Anosov; that is, $\varphi$ has distinct irrational eigenvalues $\lambda>1$ and $0<\lambda^{-1}<1$. The map $\varphi$ is an element of the mapping class group of the torus $T^2$, and the group $G_{\varphi}=\Z^2\rtimes_{\varphi}\Z$ is the fundamental group of the mapping torus of $\varphi$.

By \cite[Proposition 5.3]{FSY04}, the group $G_{\varphi}$ admits two different focal actions on $\mathbb H^2$ such that the induced diagonal action of on $\mathbb H^2\times \mathbb H^2$ is proper. Note that the map $\varphi$ is assumed to be $\left(\begin{array}{cc}
    2 & 1 \\
    1 & 1 \\
  \end{array}\right)$ in \cite{FSY04}, but their proof works for the general case without extra effort. 
Although these focal actions may not be cobounded, \cite[Lemma 4.24]{ABO19} tells us that any focal action is weakly equivalent to a cobounded focal action. In summary, we obtain 

\begin{proposition}\label{Prop: SOL}
    The fundamental group of any Anosov mapping torus has property PH$_0$ from two focal actions on $\mathbb H^2$.
\end{proposition}

To study property PH of $3$-manifold groups (Theorem \ref{IntroThm: 3Manifold}), we will use the sphere-disk decomposition and combine the results on each piece together. This relies on the following theorem about relatively hyperbolic groups. We refer the reader to \cite{Hru10} for a summary of several equivalent definitions of relatively hyperbolic groups. In particular, relatively hyperbolic groups generalize the structure of a free product acting on its Bass--Serre tree. 

\begin{theorem}\label{Thm: PHofRHG}
    Suppose that a group $G$ is hyperbolic relative to a finite collection of subgroups $\P$. If each $P\in \P$ has property PH$_0$, then $G$ has property PH$_0$. 
\end{theorem}

The proof is almost the same as the proof of \cite[Theorem 1.5]{HNY25}, modulo the estimation of the thick distance of relatively hyperbolic groups via quasi-lines \cite[Lemma 3.2, Corollary 3.3]{HNY25}, so we only give a sketch of proof that contains the main ideas. We do not need to assume residual finiteness here, also because we do not rely on \cite[Corollary 3.3]{HNY25}. 

\begin{proof}[Sketch of proof]
    By assumption, for each $P\in \P$, there are hyperbolic spaces $X_1,\dots,X_{n_P}$ such that $P$ acts on $X_i$ coboundedly and acts on $\prod_{i=1}^{n_P} X_i$ properly. Let $\XX_{P}^i = \{X_{Q}^i \mid Q\in G/P\}$ be a collection of copies of $X_i$, where $G/P$ denotes the left cosets of $P$. The group $G$ acts on $\XX_{P}^i$ by $gX_{Q}^i:=X_{gQ}^i$. Denote by $\hat G$ the coned-off graph of $G$, that is to say the metric graph obtained from a Cayley graph of $G$ by adding an edge connecting each pair of (distinct) vertices contained in the same left coset of $P\in \P$. By \cite{Far98}, $\hat G$ is hyperbolic on which $G$ acts coboundedly.
    
    Fix an orbit map $\iota_Q^i:Q\to X_Q^i$. It is conventional to define a collection of maps $\{\iota_Q^i: Q\to X_Q^i \mid Q\in G/P\}$ such that the diagram
    \[\xymatrix{
    Q \ar[r]^g \ar[d]_{\iota_Q^i} & gQ \ar[d]^{\iota_{gQ}^i} \\
    X_Q^i \ar[r]^g & X_{gQ}^i
    }\]
    commutes for any $g\in G$ and $Q\in G/P$. 
    Let $\pi_Q$ be the shortest projection to the coset $Q$ in $H$ with respect to the word metric. For any $Q\ne Q'\in G/P$, we define a projection map $\Pi_{X_{Q'}^i}(X_Q^i):=\iota_{Q'}^i(\pi_{Q'}(Q))$. 
    
    It is easy to verify that $\XX_{P}^i$ equipped with $\Pi_{X_{Q'}^i}(X_Q^i)$ satisfies the projection axioms in \cite{BBF15}. This provides a hyperbolic space $\C \XX_{P}^i$ equipped with a $G$-action. Since each $P$ acts coboundedly on $X_{P}^i$, $G$ also acts coboundely on $\C \XX_{P}^i$. Since each $P$ acts properly on $\prod_{i=1}^{n_P} X_i$, we can verify that the diagonal $G$-action on $\hat G\times \prod_{P\in\P}\prod_{i=1}^{n_P}\C \XX_{P}^i$ is proper by the distance formula for the spaces $\C \XX_{P}^i$ \cite[Theorem 4.13]{BBF15} and the distance formula for relatively hyperbolic groups \cite[Theorem 0.1]{Sis13}. 
\end{proof}

Now we can prove Theorem \ref{IntroThm: 3Manifold}.

\begin{proof}[Proof of Theorem \ref{IntroThm: 3Manifold}]
    Since $M$ is a compact, orientable $3$-manifold, it decomposes into irreducible, $\partial$-irreducible pieces $M_1\dots,M_k$ by the sphere-disk decomposition. In particular, $\pi_1(M)$ is the free product $\pi_1(M_1)* \cdots *\pi_1(M_k)* F_r$ for some free group $F_r$. Thus, the group $G=\pi_1(M)$ is hyperbolic relative to the collection $\P=\{\pi_1(M_1), \dots ,\pi_1(M_k),F_r\}$.


    Assume that there is no piece $M_i$ that supports Nil geometry. For any non-Sol piece $M_i$, $\pi_1(M_i)$ has property QT by \cite[Theorem 1.1]{HNY25}, and thus has property PH by Lemma \ref{Lem: QTImpPH}. Any Sol piece $M_i$ is a torus (semi-)bundle of Anosov type (see for example \cite[Proposition 12.7.6]{Mar23}). Hence, $\pi_1(M_i)$ also has property PH by Proposition \ref{Prop: SOL} and the commensurability invariance of property PH. 
    This shows that each peripheral subgroup $P\in \P$ has property PH. 
    
    Let $G=\pi_1(M)$. Choose a finite-index subgroup $P'$ of $P$ with property PH$_0$ for each $P\in \P$ and denote the collection of $P'$'s by $\P'$. It is well-known that $G$ is residually finite. Thus, \cite[Theorem 1.1]{Bur71} tells us that every finite-index subgroup of $P\in\P$ is separable in $G$.     
    Hence, every $P'\in \P'$ is separable in $G$, which means that the intersection of all subgroups of finite index in $G$ containing $P'$ is $P'$ itself. Therefore, there is a finite-index subgroup $G_P$ of $G$ such that $P'=G_P\cap P$. Consider the finite-index subgroup $G'=\bigcap_{P\in \P} G_P$. The group $G'$ is hyperbolic relative to $\P''=\{G'\cap P'\mid P'\in \P'\}$ and satisfies the assumption of Theorem \ref{Thm: PHofRHG}. Therefore, $G'$ has property PH$_0$ and then $G$ has property PH.
\end{proof}

\section{Central extensions with bounded Euler classes}\label{Sec: CentQH}
In this section, we will show that properties PH and QT are preserved under central extensions with bounded Euler classes (Corollary \ref{Cor: ExtHasPH}). Also, we will prove that if the total group of a central extension has property PH, then this extension must have a bounded Euler class (Corollary \ref{Eu_is_bdd}). The key tool is Proposition \ref{Prop: bddEu}, which provides several equivalent conditions for a central extension to have a bounded Euler class. 

\subsection{Geometric properties of central elements}
\begin{lemma}\label{Lem: CenterProperty}
    Let $E$ be a group acting isometrically on a hyperbolic space $X$ and $Z= Z(E)$.
    \begin{enumerate}
        \item If the action $E\curvearrowright X$ is cobounded, then any element $c\in Z$ is either elliptic or loxodromic on $X$. 
        \item If there exists an element $c\in Z$ that is loxodromic on $X$, then the $E$-action on $X$ is lineal and orientable.
    \end{enumerate}
\end{lemma}

\begin{proof}
    (1) Assume that $c$ is parabolic on $X$. Let $\xi\in \partial X$ be the unique fixed point of $c$. If there exists a loxodromic element $g\in E$ on $X$, then the equality $cgc^{-1}=g$ shows that $cg^{\pm}=g^{\pm}$ which contradicts with $|\rm{Fix(c)}|=1$. Hence, there is no loxodromic elements in $E$ on $X$. Since a horocyclic action cannot be cobounded, one gets that $E\curvearrowright X$ is a bounded action. This implies that $c$ is an elliptic element, which is also a contradiction. 
    
    (2) Since $c$ is central in $E$ and loxodromic on $X$, $E$ leaves $\Lambda\langle c\rangle=\{c^+,c^-\}$ invariant. This means the action of $E$ on $X$ is lineal. Suppose that there exists an element $b\in E$ that swaps the two points $c^+,c^-$. Then for any $o\in X$, it is clear that $|bc^no- c^nbo|\to \infty$ as $n\to \infty$. This shows that $bc^nb^{-1}\neq c^n$ for $n\gg 0$, which is a contradiction to $c\in Z$. Hence, the action of $E$ on $X$ is lineal and orientable.
\end{proof}

\begin{lemma}\label{Lem: PHToQL}
    Let $E$ be a group that has property PH$_0$ from actions on $X_1,\dots,X_n$. Let $Z\le Z(E)$. Then $Z$ is a finitely generated abelian group of rank at most $n$. Moreover, there exists a subset of quasi-lines $\{L_1,\cdots,L_r\}\subseteq \{X_1,\cdots,X_n\}$ such that the $E$-action on each $L_i$ is lineal and orientable, and $Z$-action on $\prod_{i=1}^rL_i$ is geometric. 
\end{lemma}
\begin{proof}
    From Lemma \ref{Lem: CenterProperty}, an important observation is that each $X_i$ on which $Z$ acts non-elliptically is a quasi-line and $E$ has an orientable lineal action on it. Hence, by removing all $X_i$'s on which $Z$ acts elliptically, the remaining $X_i$'s (assumed to be $X_1,\dots,X_m$) are all quasi-lines. Clearly, $Z$ has property PH$_0$ from orientable lineal actions on $X_1,\dots,X_m$. The remaining proof is completed by Lemma \ref{Lem: OriLineal}. 
\end{proof}

Next we give a necessary condition for a group to have property PH.

\begin{lemma}\label{Lem: NoPH}
    Let $G$ be a group and $c\in Z(G)$ be a central element. Suppose that a subgroup $G'\le G$ of finite index has a cobounded action on a hyperbolic space $X$. Let $c$ be a commutator of $G$ and $k>0$ such that $g^k\in G'$ for any $g\in G$, then $c^k$ acts as an elliptic element on $X$. 
\end{lemma}

\begin{proof}
    If $c$ is of finite order, then there is nothing to prove. Suppose $c$ is of infinite order. 

    \begin{claim}
        $c^{k^2m^2}$ is a commutator in $G'$ for any $m\ge 1$.
    \end{claim}
    \begin{proof}[Proof of Claim]
        Since $c$ is a commutator in $G$, there exist $g,h\in G$ such that $ghg^{-1}=ch$. Since $c\in Z(G)$, $gh^{km}g^{-1}=c^{km}h^{km}$ for all $m\ge 1$. Reformulating this equality gives that $h^{-km}gh^{km}=c^{km}g$ and then $h^{-km}g^{km}h^{km}=c^{k^2m^2}g^{km}$. Note that $g^{km}, h^{km}\in G'$ for each $m\ge 1$. Thus, $[g^{-km},h^{-km}]=c^{k^2m^2}$ is a commutator in $G'$.
    \end{proof}
    Suppose that $c^k$ is not elliptic on $X$. By Lemma \ref{Lem: CenterProperty}, the action $G'\curvearrowright X$ is orientable lineal. Hence, Lemma \ref{Lem: UniformBounded} shows that the orbit of $\langle c^{k^2}\rangle$ on $X$ is bounded, which is a contradiction to the non-ellipticity of $c^k$. 
\end{proof}

\begin{lemma}\label{Lem: CenComImpNoPH}
    Let $G$ be a group. If there exists an infinite-order element $c\in Z(G)$ such that $c$ is a commutator in $G$, then $G$ does not have property PH.
\end{lemma}
\begin{proof}
    Let $G'\le G$ be a finite-index subgroup that has property PH$_0$ from actions on $X_1,\dots,X_n$. By Lemma \ref{Lem: NoPH}, some power of $c$ is elliptic on each $X_i$ for $1\le i\le n$. This contradicts the properness of the diagonal action $G'\curvearrowright \prod_{i=1}^nX_i$ since $c$ is of infinite order.  
\end{proof}

We remark that Heisenberg groups satisfy the condition of Lemma \ref{Lem: CenComImpNoPH}.

\subsection{Characterizations of central extensions with bounded Euler classes}\label{subsec: bddEulerClass}

In order to prove our main results, we will characterize bounded Euler classes in different ways. 
Consider a central extension of groups
\begin{align}\label{CentralExt}
    1\to Z\xrightarrow{\iota}E\xrightarrow{\pi}G\to 1
\end{align}
where $Z$ is finitely generated.
Let $[\omega]\in H^2(G,Z)$ be the Euler class of this central extension. In what follows, we will identify $Z$ with $\iota(Z)$. For any $g\in E$, we write $\bar g$ to mean $\pi(g)\in G$. 

According to the structure theorem of finitely generated abelian groups, there exists $r\ge 0$ and a finite abelian group $F$ such that $Z\cong \Z^r\bigoplus F$. Let $\{c_1,\dots, c_r\}$ be a standard generating set for $\Z^r$. 

\begin{proposition}\label{Prop: bddEu}
    For the central extension (\ref{CentralExt}), the following are equivalent:
    \begin{enumerate}
        \item $[\omega]$ is bounded. \label{eq1}
        \item The central extension (\ref{CentralExt}) quasi-splits. \label{eq2}
        \item The central extension (\ref{CentralExt}) admits a quasi-retraction. \label{eq3}
        \item There exist quasi-lines $L_1,\cdots,L_r$ on which $E$ acts lineally and orientably such that the diagonal action of $Z$ on $\prod_{i=1}^rL_i$ is geometric. \label{eq6}
        \item The central extension (\ref{CentralExt}) virtually quasi-splits. \label{eq8}
    \end{enumerate}
\end{proposition}

Shortly before we finished this paper, Fournier-Facio, Mangioni, and Sisto also obtained part of Proposition \ref{Prop: bddEu} independently with similar methods (see \cite[Proposition 2.9]{FMS25}).

\begin{proof}
    The equivalence of $(\ref{eq1})$ and $(\ref{eq2})$ is given by \cite[Lemma 2.9]{FK16}. The equivalence of $(\ref{eq2})$ and $(\ref{eq3})$ is given by \cite[Theorem 1.5]{Wan25}.  ``$(\ref{eq2}) \Rightarrow (\ref{eq8})$'' is obvious. For ``$(\ref{eq3}) \Rightarrow (\ref{eq6})$'', note that by composing the quasi-retraction with projections to each $\Z$-factor, there exist quasimorphisms $\phi_i: E\to \Z$ for $i=1,\dots,r$, such that $\phi_i$ is unbounded on $\langle c_i \rangle$ and bounded on $\langle c_j \rangle$ whenever $i\ne j$. Thus, (\ref{eq6}) is given by Lemma \ref{Lem: TFAE}. 
    Now we only need to prove ``$(\ref{eq6}) \Rightarrow (\ref{eq1})$'' and ``$(\ref{eq8}) \Rightarrow (\ref{eq2})$''.

    For $(\ref{eq6}) \Rightarrow (\ref{eq1})$: We divide the proof into three steps.
    
    \textbf{Step 1: we construct a quasi-homomorphism $\varphi$ from $E$ to $Z$.}
    
    Since the action of $Z$ on $\prod_{i=1}^rL_i$ is geometric, there exists at least one loxodromic element in $Z$ on each $L_i$. It follows from Lemma \ref{Lem: CenterProperty} that $E$ acts lineally and orientably on each $L_i$.  By Lemma \ref{Lem: TFAE}, each orientable lineal action $E\curvearrowright L_i$ gives a homogeneous quasimorphism $\phi_i:E\to \R$. Define $\Phi: E\to \R^r$ by $\Phi(g):=(\phi_1(g),\cdots,\phi_r(g))^T$ for any $g\in E$. As a product of homogeneous quasimorphisms, $\Phi$ is a homogeneous quasi-homomorphism. Denote $K=\sup_{g,h\in E}\|\Phi(gh)-\Phi(g)-\Phi(h)\|_1$. Then for any $c=c_1^{n_1}\cdots c_r^{n_r}\in Z$, 
    \begin{equation}\label{Equ: theta}
        \|\Phi(c)-(\sum_{i=1}^r\Phi(c_i^{n_i}))\|_1=\|\Phi(c)-(\sum_{i=1}^rn_i\Phi(c_i))\|_1=\|\Phi(c)-\Theta(n_1,\cdots,n_r)^T\|_1\le (r-1)K
    \end{equation}
    where $\Theta:=(\Phi(c_1),\cdots,\Phi(c_r))$ is a matrix in $M_r(\R)$. Since the action of $Z$ on $\prod_{i=1}^rL_i$ is geometric, $\Theta\in \gl_r(\R)$. To obtain a quasi-homomorphism from $E$ to $Z$, we need to composite another two quasi-homomorphisms.

    For any $\alpha=(x_1,\ldots,x_r)^T\in \R^r$, denote $\lfloor\alpha\rfloor:=(\lfloor x_1\rfloor,\ldots,\lfloor x_r\rfloor)^T\in \Z^r$. Define two quasi-homomorphisms $\eta: \R^r\to \Z^r$ and $\rho: \Z^r\to Z$ by $\eta(\alpha):=\lfloor\Theta^{-1}\alpha\rfloor$ for any $\alpha\in \R^r$ and $\rho(n_1,\cdots,n_r):=c_1^{n_1}\cdots c_r^{n_r}$ for any $n_1,\ldots,n_r\in \Z$. Note that $\rho$ is a group isomorphism between $\Z^r$ and a finite-index subgroup of $Z$. Let $\varphi:=\rho\circ \eta\circ \Phi: E\to Z$. As a composition of quasi-homomorphisms, $\varphi$ is a quasi-homomorphism from $E$ to $Z$. 
    
    \textbf{Step 2: we show $\varphi|_Z$ is close to the identity map.}

    For any $c=c_1^{n_1}\cdots c_r^{n_r}\in Z$, we denote $\alpha:=\Phi(c)-\Theta(n_1,\cdots,n_r)^T$. By compositing $\eta$ on the left, we have $$\eta(\alpha)=\lfloor\Theta^{-1}\alpha\rfloor=\eta\circ \Phi(c)-(n_1,\cdots,n_r)^T.$$
    Next, by compositing $\rho$ on the left, we have $$\rho(\lfloor\Theta^{-1}\alpha\rfloor)=c^{-1}\varphi(c).$$ Note that Formula (\ref{Equ: theta}) shows $\|\alpha\|_1\le (r-1)K$. Hence, the above equality shows that there exists a finite subset $A\subset Z$ such that $c^{-1}\varphi(c)\in A$ for any $c\in Z$.
    
    \textbf{Step 3: we construct a bounded representative of $[\omega]$.}

    For any $g\in E$, $g=\varphi(g)\cdot (\varphi(g)^{-1}g)$. Since $\varphi$ is a quasi-homomorphism and $\varphi(g)^{-1}\in Z$, one has
    
    $$\varphi(\varphi(g)^{-1}g)\in \varphi(\varphi(g)^{-1})\varphi(g)\Delta(\varphi)\in \varphi(g)^{-1}\varphi(g)A\cdot \Delta(\varphi)=A\cdot \Delta(\varphi).$$
    
    Denote $B=A\cdot \Delta(\varphi)$. The above equality shows that for any $g\in E$, $\varphi(g)^{-1}g\in \varphi^{-1}(B)$. Since $\varphi(g)\in Z=\ker(\pi)$,  $\pi(\varphi(g)^{-1}g)=\pi(g)$. This allows us to pick a section $s: G\to E$ such that $s(G)\subseteq \varphi^{-1}(B)$. Let $\omega_s: G\times G\to Z$ given by $\omega_s(g,h)=s(g)s(h)s(gh)^{-1}$ be the corresponding cocycle. Since the image of $\omega_s$ lies in $Z$, \textbf{Step 2} implies the equivalence between boundedness of $\omega_s$ and boundedness of $\varphi(\omega_s)$. The latter is clearly bounded since $s(G)\subseteq \varphi^{-1}(B)$. Therefore, the Euler class $[\omega_s]$ is bounded.

    For $(\ref{eq8}) \Rightarrow (\ref{eq2})$: By definition, there exists a finite-index subgroup $E'\le E$ such that the following central extension $1\to Z'\to E'\xrightarrow{\pi}G'\to 1$ quasi-splits where $Z'=Z\cap E'$ and $G'=\pi(E')$. By the equivalence of (\ref{eq2}) and (\ref{eq3}), there exists a quasi-retraction $\Phi':E'\to Z'$. Assume $Z'=\Z^r$ that embeds into $R=\R^r$ canonically. We consider $\Phi'$ as a quasi-homomorphism with values in $R$. By Lemma \ref{Lem: TransferMap}, $\Phi': E'\to R$ induces a quasi-homomorphism $\hat \T(\Phi'): E\to R$ whose restriction on $Z'$ is the identity map. 
    
    Denote  $\bar E=E/Z', \bar Z=Z/Z'$. Then $G=E/Z\cong \frac{E/Z'}{Z/Z'}=\bar E/\bar Z$. The original central extension $1\to Z\to E\xrightarrow{\pi} G\to 1$ is splitted into the following two central extensions
    \begin{equation}\label{Equ: SplCenExt1}
        1\to Z'\to E\xrightarrow{\pi_1} \bar E\to 1
    \end{equation}
    and 
    \begin{equation}\label{Equ: SplCenExt2}
        1\to \bar Z\to \bar E\xrightarrow{\pi_2} G\to 1.
    \end{equation}
    Since the map $\hat \T(\Phi'): E\to R$ is a quasi-homomorphism whose restriction on $Z'$ is the identity map, the equivalence of (\ref{eq2}) and (\ref{eq3}) shows that the central extension (\ref{Equ: SplCenExt1}) quasi-splits. Note that $\bar Z$ is finite since $Z'$ is of finite index in $Z$. Thus any normalized section of the central extension (\ref{Equ: SplCenExt2}) is a quasi-splitting. 

    Let $s_1: \bar E\to E$ and $s_2: G\to \bar E$ be two quasi-splittings. Define a map $s: G\to E$ as the composite map $s_1\circ s_2$. Then $s$ is the desired quasi-splitting from $G$ to $E$. 

\end{proof}

\begin{remark}\label{Rmk: Phi(s(x))=1}
     (1) Suppose that $[\omega]$ is bounded. Let $s: G\to E$ be the quasi-splitting given by (\ref{eq2}). From the proof of \cite[Theorem 3.15]{Wan25}, the map $\phi: E\to Z$ defined by $\phi(g)=gs(\bar g)^{-1}$ is a desired quasi-retraction for (\ref{eq3}). Note that $\phi(g)=1$ for any $g\in s(G)$. 
     According to Remark \ref{Rmk: STL}, this implies that $s(G)$ has a bounded orbit on each quasi-line $L_i$ obtained in (\ref{eq6}).

     (2) When the group extension is not a central extension, Items (\ref{eq2}), (\ref{eq3}), (\ref{eq8}) in Proposition \ref{Prop: bddEu} may not be equivalent anymore although the relation ``(\ref{eq3}) $\Rightarrow$ (\ref{eq2}) $\Rightarrow$ (\ref{eq8})'' always holds. In \cite{Wan25}, the second author shows that a general group extension admits a quasi-retraction if and only if it quasi-splits and the image of this quasi-splitting almost commutes with the normal subgroup. Besides, there exists a group extension which virtually quasi-splits but not quasi-splits. For example, the surjective homomorphism from $F_2$ to the infinite dihedral group $D_{\infty}$ gives a short exact sequence of groups $1\to H\to F_2\to D_{\infty}\to 1$. Since $D_{\infty}$ is virtually cyclic, it is easy to see this extension virtually quasi-splits. However, it does not quasi-split since any unbounded quasi-homomorphism to $F_2$ is either a homomorphism or a quasimorphism with images in a cyclic subgroup \cite[Theorem 4.1]{FK16}.
\end{remark}

The following corollary is a complement to Proposition \ref{Prop: bddEu} (\ref{eq6}). 

\begin{corollary}\label{Cor: bddEuCor}
    If $[\omega]$ is bounded and $G$ admits a proper action on a metric space $X$, then the induced action of $E$ on $X\times \prod_{i=1}^r L_i$ by $g\cdot (x, y_1,\ldots, y_r):=(\bar gx, gy_1,\ldots, gy_r)$ is proper where $L_1,\cdots,L_r$ are provided by Proposition \ref{Prop: bddEu} (\ref{eq6}). If in addition that $G$ is generated by a finite set $S$, then the induced action of $E$ on $\G(G,S)\times \prod_{i=1}^r L_i$ is geometric.
\end{corollary}

\begin{proof}
    By Proposition \ref{Prop: bddEu} (\ref{eq2}), there exists a quasi-splitting $s: G\to E$. Any group element $g\in E$ has a decomposition $g=g_Zg_G$, where $g_Z=g\cdot s(\bar g)^{-1}$ and $g_G=s(\bar g)$. Note that $g_Z\in Z$ and $\bar g=\bar{g_G}$. Fix a basepoint $(x, o_1,\ldots, o_r)\in X\times \prod_{i=1}^r L_i$. For any $C>0$, we denote $B(C):=\{g\in E: |x-\bar gx|+\sum_{i=1}^r|o_i-go_i|\le C\}$. Let $g\in B(C)$. There are only finitely many choices of $g_G=s(\bar g)$ in $E$ since $\{\bar g\in G : |x-\bar gx|\le C\}$ is finite. Now we fix such a choice of $g_G$. By triangle inequality, $$\sum_{i=1}^r|o_i-g_Zo_i|\le \sum_{i=1}^r(|o_i-g_Zg_Go_i|+|o_i-g_Go_i|)=\sum_{i=1}^r|o_i-go_i|+\sum_{i=1}^r|o_i-g_Go_i|\le C+\sum_{i=1}^r|o_i-g_Go_i|.$$
    Since the action $Z\curvearrowright \prod_{i=1}^r L_i$ is geometric, there are only finitely many choices of $g_Z$. This shows that $B(C)$ is finite for any $C>0$, which gives the properness of the diagonal action $E\curvearrowright X\times \prod_{i=1}^r L_i$.

    Suppose that $G$ is generated by a finite set $S$. By the above arguments, we only need to show that the induced action $E\curvearrowright \G(G,S)\times \prod_{i=1}^r L_i$ is cobounded. Fix a basepoint $(1, o_1,\ldots, o_r)\in \G(G,S)\times \prod_{i=1}^r L_i$ and an arbitrary point $(x, y_1,\ldots, y_r)\in \G(G,S)\times \prod_{i=1}^r L_i$. Since $Z\curvearrowright \prod_{i=1}^r L_i$ is geometric, there exists $c\in Z$ such that $c\cdot(o_1,\cdots,o_r)$ is uniformly close to $(y_1,\cdots, y_r)$. Let $g=s(x)c\in E$. By Remark \ref{Rmk: Phi(s(x))=1}, $s(G)$ has a bounded orbit on $\prod_{i=1}^r L_i$. Hence, $g\cdot(1,o_1,\cdots,o_r)$ is uniformly close to $(x,y_1,\cdots, y_r)$, which completes the proof of coboundedness.
\end{proof}

\begin{corollary}\label{Cor: ExtHasPH}
    Assume that $[\omega]$ is bounded. If $G$ has property PH or QT, then $E$ has property PH or QT respectively. 
\end{corollary}

\begin{proof}
    Now assume that $G$ has property PH. Then $G$ has a finite-index subgroup $G'$ with property PH$_0$. Denote $E'=\pi^{-1}(G')$, which is a finite-index subgroup of $E$. Note that $Z$ is contained in the center of $E'$. Hence, one has the following central extension  $$1\to Z\to E'\xrightarrow{\pi'} G'\to 1$$ where $\pi'=\pi|_{E'}$. The corresponding Euler class $[\omega']$ is bounded since $[\omega]$ is bounded. Therefore, $E'$ has property PH$_0$ by Corollary \ref{Cor: bddEuCor}, and the conclusion follows.

    Next, we assume that $G$ is finitely generated and has property QT. By the Svarc-Milnor lemma, the geometric action $E\curvearrowright \G(G,S)\times \prod_{i=1}^rL_i$ in Corollary \ref{Cor: bddEuCor} gives a finite generating set $F$ of $E$ such that the orbit map $\G(E,F)\to \G(G,S)\times \prod_{i=1}^rL_i$ is a quasi-isometry. Composing this equivariant quasi-isometry with the QT embedding of $G$, one obtains that $E$ also has property QT.
\end{proof}

\begin{corollary}\label{Eu_is_bdd}
    If $E$ has property PH, then $[\omega]$ is bounded. 
\end{corollary}

\begin{proof}
    By Proposition \ref{Prop: bddEu}, it suffices to show the central extension (\ref{CentralExt}) virtually quasi-splits. 
    
    Let $E'\le E$ be a finite-index subgroup that has property PH$_0$. Let $Z'=Z\cap E'$ and $G'=\pi(E')$. Consider the following central extension $1\to Z'\to E'\xrightarrow{\pi}G'\to 1$. By Lemma \ref{Lem: PHToQL}, there exist quasi-lines $L_1,\cdots,L_r$ on which $E'$ acts lineally such that the diagonal action of $Z'$ on $\prod_{i=1}^rL_i$ is geometric. This shows that the above central extension satisfies Proposition \ref{Prop: bddEu} (\ref{eq6}). As a result of Proposition \ref{Prop: bddEu}, the above central extension quasi-splits, which amounts to saying that the original central extension (\ref{CentralExt}) virtually quasi-splits.
\end{proof}

Note that Corollary \ref{Cor: ExtHasPH} gives one direction of Theorem \ref{MainThm1} and Corollary \ref{Eu_is_bdd} gives part of the other direction. We will complete the remaining proof in Section \ref{Sec: PHUnderQuotient}.

\section{Central quotients}\label{Sec: PHUnderQuotient}
Keep the setup of Subsection \ref{subsec: bddEulerClass}: $1\to Z\xrightarrow{\iota}E\xrightarrow{\pi}G\to 1$ is a central extension of groups where $Z$ is finitely generated and $[\omega]\in H^2(G,Z)$ is the corresponding Euler class. We identify $\iota(Z)$ with $Z$ and write $\bar g$ to mean $\pi(g)\in G$.

\subsection{Constructing hyperbolic spaces for central quotients}\label{subsec: ConsHypQuoSpace}

The first result of this section is the following proposition, which provides the proof idea of Theorem \ref{MainThm1}. 
\begin{proposition}\label{Prop: QuoHasPH}
    If $E$ has property PH$_0$, then $G$ has property PH$_0$.
\end{proposition}

Let $X$ be a geodesic metric space and $\Gamma$ be a group acting isometrically on $X$. Let $H\lhd \Gamma$. Recall that there is a well-defined quotient metric on $X/H$, and $\Gamma/H$ acts isometrically on $X/H$.


\begin{lemma}\label{Lem: QuoHypSpace}
Let $\Gamma$ be a group acting coboundedly on a hyperbolic space $X$. Let $c$ be a central element of $\Gamma$. If $c$ is elliptic on $X$, then $X/\langle c\rangle$ is a hyperbolic space quasi-isometric to $X$. If $c$ is loxodromic on $X$, then $X/\langle c\rangle$ is a bounded space.
\end{lemma}
\begin{proof}
If $c$ is elliptic on $X$, then the conclusion can be deduced from \cite[Lemma 4.10]{BFFG24}. If $c$ is loxodromic on $X$, then Lemma \ref{Lem: CenterProperty} shows that $X$ is a quasi-line. So $X/\langle c\rangle$ is a bounded space.
\end{proof}

From now on, we assume that $E$ has property PH$_0$ from actions on $X_1,\dots,X_n$.

According to the structure theorem of finitely generated abelian groups, there exist $r\in \N$ and a finite abelian group $F$ such that $Z\cong \Z^r\bigoplus F$. By Lemma \ref{Lem: QuoHypSpace}, each quotient space $X_i/F$ is still hyperbolic and the induced diagonal action $E/F\curvearrowright \prod_{i=1}^nX_i/F$ is still proper. Based on this fact, we can assume that $Z$ is infinite (i.e. $r\ge 1$) and torsion-free. Let $\{c_1,\ldots,c_r\}$ be a standard generating set of $Z$. 

For simplicity, we denote  $\bar X_i:=X_i/Z$ for each $1\le i\le n$. Up to re-ordering, let $\bar X_1, \ldots, \bar X_m(1\le m\le n)$ be all bounded quotient spaces. We collect some basic properties about $\bar X_i(1\le i\le n)$.

\begin{lemma}\label{Lem: QuoSpaProperty}
    ~
    \begin{enumerate}
        \item\label{it1} For $1\le i\le m$, $X_i$ is a quasi-line and the action of $E$ on $X_i$ is orientable lineal.
        \item\label{it2} For each $1\le i\le n$, $\bar X_i$ is hyperbolic and the following are equivalent:
        \begin{itemize}
            \item[(i)] $m+1\le i\le n$, i.e. $\bar X_i$ is unbounded;
            \item[(ii)] $Z$ has an elliptic action on $X_i$;
            \item[(iii)] $X_i$ is quasi-isometric to $\bar X_i$.
        \end{itemize}
        \item\label{it3} The integer $m$ satisfies $r\le m\le n$. Moreover, if $m=n$, then $E$ and $G$ are two finitely generated elementary (i.e. virtually abelian) groups.
    \end{enumerate}
\end{lemma}
\begin{proof}
    (1) This follows from Lemma \ref{Lem: CenterProperty}. 

    (2) This follows easily from Lemma \ref{Lem: QuoHypSpace}. 

    (3) Since each $\bar X_i$ is unbounded for $m+1\le i\le n$, it follows from Item (\ref{it2}) that $Z$ has an elliptic action on $\prod_{i=m+1}^n X_i$. Thus the proper diagonal action $Z\curvearrowright \prod_{i=1}^nX_i$ passes to a proper diagonal action $Z\curvearrowright \prod_{i=1}^mX_i$. By Item (\ref{it1}), these orientable lineal actions $Z\curvearrowright X_1,\cdots,Z\curvearrowright X_m$ give the PH$_0$ property of $Z$. Then it follows from Lemma \ref{Lem: OriLineal} that $r\le m$. Moreover, if $m=n$, then all $\bar X_i(1\le i\le n)$ are bounded. By Item (\ref{it1}) again, $E$ has property PH$_0$ from orientable lineal actions. As a result of Lemma \ref{Lem: OriLineal}, $E$ is a finitely generated elementary group. As a quotient of $E$, $G$ is also finitely generated and elementary.
\end{proof}

Now we are in a position to prove Proposition \ref{Prop: QuoHasPH}. 

\begin{proof}[Proof of Proposition \ref{Prop: QuoHasPH}]
    
    We will divide the proof into two steps.

    \textbf{Step 1: we construct auxiliary quasi-lines on which $G$ acts lineally.} 

   Recall from Lemma \ref{Lem: QuoSpaProperty} that each $X_i(1\le i\le m)$ is a quasi-line and the action of $E$ on $X_i$ is orientable lineal. By \cite[Lemma 3.7]{Man06}, there exists a uniform constant $\epsilon\ge 0$ such that each $X_i$ is $(1,\epsilon)$-quasi-isometric to $\R$. Denote $f_i: X_i\to \R$ as a $(1,\epsilon)$-quasi-isometry. It is easy to see the map $\chi_i: E\to \R$ given by $\chi_i(g):=f_i\circ g\circ f_i^{-1}(0)$ is an unbounded quasimorphism for $1\le i\le m$. 
   
   Since $E$ has property PH$_0$, Corollary \ref{Eu_is_bdd} shows that $[\omega]$ is bounded. This is equivalent to the existence of a quasi-splitting $s: G\to E$ by Proposition \ref{Prop: bddEu}. Then the composite map $\chi_i'=\chi_i\circ s: G\to \R$ is also a (possibly bounded) quasimorphism. WLOG, we assume that $\chi_1',\ldots,\chi_k'$ are all unbounded quasimorphisms.  If $k=0$, then there is nothing to supplement in this step. If $1\le k\le m$, then \cite[Lemma 4.15]{ABO19} shows that for each $1\le i\le k$, there exists a generating set $S_i$ of $G$ such that the map $\chi_i'$ gives a quasi-isometry between $\G(G,S_i)$ and $\R$. Let $\lambda\ge 1, \epsilon'\ge 0$ be two constants such that $\chi_i': \G(G,S_i)\to \R$ is a $(\lambda,\epsilon')$-quasi-isometry for each $1\le i\le k$. 

   \textbf{Step 2: we show the diagonal action $G\curvearrowright \prod_{i=1}^k\G(G,S_i)\times \prod_{i=m+1}^n\bar X_i$ is proper.} 
   
   Fix a basepoint $(x_1,\cdots,x_n)\in \prod_{i=1}^nX_i$ with $x_i\in f_i^{-1}(0)$ for $1\le i\le m$. Let $\bar x_i\in \bar X_i$ be the image of $x_i\in X_i$ under the quotient map.

   \begin{claim}
       There exists a constant $M>0$ such that $$\sum_{i=1}^k|\bar g|_{S_i}+\sum_{i=m+1}^n|\bar x_i-\bar g\bar x_i|\ge \lambda^{-1}(\sum_{i=1}^n|x_i-s(\bar g)x_i|)-M$$ for any $\bar g\in G$.
   \end{claim}
   \begin{proof}[Proof of Claim]
       We prove the claim in three cases according to the range of $i$.
       
       Firstly, we estimate $|x_i-s(\bar g)x_i|$ for $1\le i\le k$. Note that each $\chi_i'$ is a $(\lambda,\epsilon')$-quasi-isometry. Hence, $|\chi_i'(\bar g)|=|\chi_i(s(\bar g))|\le \lambda |\bar g|_{S_i}+\epsilon'$. By definition of $\chi_i$, $\chi_i(s(\bar g))=f_i\circ s(\bar g)\circ f_i^{-1}(0)=f_i(s(\bar g)x_i)$. Since each $f_i$ is a $(1,\epsilon)$-quasi-isometry and $f_i(x_i)=0$, one gets that $$|x_i-s(\bar g)x_i|\le |f_i(x_i)-f_i(s(\bar g)x_i)|+\epsilon=|\chi_i(s(\bar g))|+\epsilon=|\chi_i'(\bar g)|+\epsilon\le \lambda |\bar g|_{S_i}+\epsilon+\epsilon'.$$

       Secondly, we estimate $|x_i-s(\bar g)x_i|$ for $k+1\le i\le m$. In this case, each $\chi_i': G\to \R$ is a bounded quasimorphism. The same arguments as above show that $$|x_i-s(\bar g)x_i|\le |f_i(x_i)-f_i( s(\bar g)x_i)|+\epsilon=|\chi_i'(\bar g)|+\epsilon\le M_1$$ for some $M_1>0$ independent of $\bar g$.
       
       Finally, we estimate $|x_i-s(\bar g)x_i|$ for $m+1\le i\le n$. In this case, $Z$ has an elliptic action on each $X_i$ since $\bar X_i(m+1\le i\le n)$ are still unbounded hyperbolic spaces. Thus there exists a constant $M_2>0$ such that $|x_i-cx_i|\le M_2$ for any $c\in Z$. By the definition of quotient metric, one has that $|\bar x_i-\bar g\bar x_i|=|x_i-s(\bar g)cx_i|$ for some $c\in Z$. Therefore, by triangle inequality, $$|x_i-s(\bar g)x_i|\le | x_i-s(\bar g)cx_i|+|x_i-cx_i|= |\bar x_i-\bar g\bar x_i|+|x_i-cx_i|\le |\bar x_i-\bar g\bar x_i|+M_2.$$

       In conclusion, we get that 
       \begin{align*}
           & \sum_{i=1}^k|\bar g|_{S_i}+\sum_{i=m+1}^n|\bar x_i-\bar g\bar x_i| \\
           \ge & \lambda^{-1}\sum_{i=1}^k|x_i-s(\bar g) x_i|-k\lambda^{-1}(\epsilon+\epsilon')+\sum_{i=m+1}^n|x_i-s(\bar g) x_i|-M_2(n-m) \\
           \ge & \lambda^{-1}\sum_{i=1}^k|x_i-s(\bar g)x_i|+\sum_{i=k+1}^n|x_i-s(\bar g) x_i|-k\lambda^{-1}(\epsilon+\epsilon')-M_2(n-m)-M_1(m-k) \\
           \ge & \lambda^{-1}(\sum_{i=1}^n|x_i-s(\bar g) x_i|)-M
       \end{align*}
       where $M=k\lambda^{-1}(\epsilon+\epsilon')+M_2(n-m)+M_1(m-k)$.
   \end{proof}
   Since the diagonal action $E\curvearrowright \prod_{i=1}^nX_i$ is proper, it follows from the above Claim that the diagonal action $G\curvearrowright \prod_{i=1}^k\G(G,S_i)\times \prod_{i=m+1}^n\bar X_i$ is also proper. This gives the PH$_0$ property of $G$.
\end{proof}

\subsection{Proof of Theorem \ref{MainThm1}}

    Finally, we can complete the proof of Theorem \ref{MainThm1}.

\begin{proof}[Proof of Theorem \ref{MainThm1}]

    Recall that Corollary \ref{Cor: ExtHasPH} gives the backward direction of Theorem \ref{MainThm1}. 
    
    For the forward direction, let $E'$ be a finite-index subgroup of $E$ so that $E'$ has property PH$_0$. First note that $Z\cap E'$ is finitely generated by Lemma \ref{Lem: PHToQL}, so $Z$ is finitely generated. 
    Then Corollary \ref{Eu_is_bdd} shows that the Euler class is bounded. The remaining proofs are divided into two parts.

    \textbf{Part I: we show that if $E$ has property PH, then $G$ has property PH.}
    
    Consider the following central extension of groups $$1\to Z'\to E'\to G'\to 1$$ 
    where $Z'=Z\cap E'$ and $G'$ is the image of $\pi|_{E'}$, which is a finite-index subgroup of $G$. Let $[\omega']$ be the corresponding Euler class. By Proposition \ref{Prop: QuoHasPH}, $G'$ has property PH$_0$. Hence, $G$ has property PH since $[G:G']<\infty$. 

    \textbf{Part II: we show that if $E$ has property QT, then $G$ has property QT.}

    \textbf{Step 1: if $E$ has property QT$_0$, then $G$ has property QT$_0$.}

    According to the proof of Lemma \ref{Lem: QTImpPH}, there exist finitely many quasi-trees $T_i(1\le i\le n)$ on which $E$ acts coboundedly and the orbit map from $E$ to the product $X=\prod_{i=1}^nT_i$ equipped with $\ell^1$-metric is a quasi-isometric embedding. Since $E$ has property QT$_0$ and thus PH, $[\omega]$ is bounded by Corollary \ref{Eu_is_bdd}. Let $s: G\to E$ be a quasi-splitting provided by Proposition \ref{Prop: bddEu}. Fix a finite generating set $F$ of $G$ and a finite generating set $A$ of $Z$. It is easy to see that $F'=A\cup s(F)$ is a finite generating set of $E$. Fix a basepoint $x=(x_1,\cdots,x_n)\in \prod_{i=1}^nT_i$. The QT$_0$ property of $E$ shows that the orbit map $\G(E,F')\to \prod_{i=1}^nT_i, g\mapsto gx=(gx_1,\cdots,gx_n)$ is a $(\lambda_0,\epsilon_0)$-quasi-isometric embedding for some $\lambda_0\ge 1, \epsilon_0\ge 0$. That is,
    \begin{equation}\label{Equ: QIE}
        \lambda_0^{-1}|g|_{F'}-\epsilon_0\le |x-gx|=\sum_{i=1}^n|x_i-gx_i|\le \lambda_0|g|_{F'}+\epsilon_0
    \end{equation}
    for any $g\in E$. 

    Let $\bar T_i=T_i/Z$ be the quotient space for $1\le i\le n$. Similar to the proof of Proposition \ref{Prop: QuoHasPH}, we assume that there exist $1\le k\le m\le n$ such that $\bar T_i(1\le i\le m)$ are all bounded quotient spaces and $\chi_i'=\chi_i\circ s: G\to \R(1\le i\le k)$ are all unbounded quasimorphisms. Then there exist generating sets $S_i$ of $G$ such that $[S_i]\in \H_{\ell}^+(G)$ for $1\le i\le k$ and $G$ admits a proper diagonal action on $\prod_{i=1}^k\G(G,S_i)\times \prod_{i=m+1}^n\bar T_i$. Moreover, the Claim in the proof of Proposition \ref{Prop: QuoHasPH} shows that there exists $M>0$ such that $$\sum_{i=1}^k|\bar g|_{S_i}+\sum_{i=m+1}^n|\bar x_i-\bar g\bar x_i|\ge \lambda^{-1}(\sum_{i=1}^n|x_i-s(\bar g)x_i|)-M$$ for any $\bar g\in G$.

    Together with the inequality (\ref{Equ: QIE}), one has that
    \begin{align*}
         \sum_{i=1}^k|\bar g|_{S_i}+\sum_{i=m+1}^n|\bar x_i-\bar g\bar x_i| \ge & \lambda^{-1}(\sum_{i=1}^n|x_i-s(\bar g)x_i|)-M\\ \ge & (\lambda\lambda_0)^{-1}|s(\bar g)|_{F'}-\lambda^{-1}\epsilon_0-M \ge (\lambda\lambda_0)^{-1}|\bar g|_{F}-\lambda^{-1}\epsilon_0-M
    \end{align*}
    where the last inequality holds since $|s(\bar g)|_{F'}\ge |\bar g|_F$.
    
    On the other hand, it follows from the triangle inequality that $$\sum_{i=1}^k|\bar g|_{S_i}+\sum_{i=m+1}^n|\bar x_i-\bar g\bar x_i|\le \lambda' |\bar g|_F$$ where $\lambda'=\max_{t\in F}(\sum_{i=1}^k|t|_{S_i}+\sum_{i=m+1}^n|\bar x_i-t\bar x_i|)$.

    By setting $\lambda_1=\max\{\lambda',\lambda\lambda_0\}$ and $\epsilon_1=\lambda^{-1}\epsilon_0+M$, we have shown that the orbit map $\G(G,F)\to \prod_{i=1}^k\G(G,S_i)\times \prod_{i=m+1}^n\bar T_i, \bar g\mapsto (\bar g, \cdots \bar g, \bar g\bar x_{m+1}, \cdots, \bar g\bar x_n)$ is a $(\lambda_1,\epsilon_1)$-quasi-isometric embedding. This gives the QT$_0$ property of $G$ since each $\G(G,S_i)$ is a quasi-line. 

    \textbf{Step 2: if $E$ has property QT, then $G$ has property QT.}

    Suppose that $E$ has property QT. Then $E$ contains a finite-index subgroup $E'$ that has property QT$_0$. Denote $Z'=Z\cap E'$ and $G'=\pi(E')$. Since $[E: E']<\infty$, it follows that $[Z: Z']<\infty$ and $[G: G']<\infty$. Now we focus on the following central extension  $$1\to Z'\to E'\xrightarrow{\pi'} G'\to 1.$$ Let $[\omega']$ be the corresponding Euler class. The same arguments in \textbf{Step 1} shows that $G'$ has property QT$_0$ and thus QT. Since property QT is a commensurability invariant, we get that $G$ has property QT. 
\end{proof}

\subsection{Mapping class groups}
    Now we can prove Theorem \ref{IntroThm: MCG} and Theorem \ref{IntroThm: Multicurve}. 
    Let $\Sigma$ be a finite-type surface, possibly with boundary $\partial\Sigma$. 
    The \emph{mapping class group} $\mcg(\Sigma)$ is the group of homeomorphisms of $\Sigma$ restricting to the identity on $\partial\Sigma$, modulo isotopy. If there are punctures $P=\{p_1,\dots,p_n\}$ on $\Sigma$, we define $\mcg(\Sigma,P)$ to be the finite-index subgroup of $\mcg(\Sigma)$ that fixes each puncture. 

    Let $\Sigma$ be a finite-type surface. Let $V\subset \Sigma$ be a (possibly disconnected) closed subsurface. Let $V^{\orth}$ be the closure of $\Sigma-V$ in $\Sigma$. We say that $V$ is \emph{essential} if 
    \begin{enumerate}
        \item No component of $V$ or $V^{\orth}$ is a closed disk.
        \item Annular components of $V$ are pairwise non-homotopic, and the same holds for $V^{\orth}$.
    \end{enumerate}
    In the following, let $V\subset \Sigma$ be an essential subsurface, and let $\hat V$ be the \emph{capped} surface obtained by gluing a once-punctured disk to each boundary component of $V$. There is an induced homomorphism $\eta_V:\mcg(V)\to \mcg(\Sigma)$. If $V$ is a closed annulus, then $\eta_V$ is always injective. 
    The following theorem determines the kernel of $\eta_V$ when $V$ is not a closed annulus. For any essential simple closed curve $\gamma\subset \Sigma$, we denote by $T_{\gamma}\in \mcg(\Sigma)$ the Dehn twist with respect to $\gamma$.

    \begin{theorem}\cite[Theorem 3.18]{FM12}\label{theorem: Inclusion}
        Assume that $V$ is non-annular. Let $\alpha_1,\dots,\alpha_m$ denote the boundary components of $V$ that bound once-punctured disks in $V^{\orth}$, and let $\{\beta_1^+,\beta_1^-\},\dots,\{\beta_n^+,\beta_n^-\}$ denote the pairs of boundary components of $V$ that bound annuli in $V^{\orth}$. Then $\ker(\eta_V)$ is the free abelian group generated by $T_{\alpha_1},\dots, T_{\alpha_m}$ and $T_{\beta_1^+}T_{\beta_1^-}^{-1},\dots,T_{\beta_n^+}T_{\beta_n^-}^{-1}$.
    \end{theorem}

    Let $V$ be a finite-type surface with nonempty boundary components $\gamma_1,\dots,\gamma_n$. 
    Let $\hat V$ be the capped surface obtained by gluing a once-punctured disk with puncture $p_i$ to each $\gamma_i$, and let $P=\{p_1,\dots,p_n\}$. 
    There is a central extension called the \emph{capping sequence} 
    \begin{align*}
        1\to T_{\Gamma} \to \mcg(V)\xrightarrow{\pi} \mcg(\hat V,P)\to 1,
    \end{align*}
    where $T_{\Gamma}$ is the free abelian subgroup generated by Dehn twists $T_{\gamma_i}$ around $\gamma_i$ for $i=1,\dots,n$. 
    By \cite[Propositon 6.4]{FK16}, the capping sequence has a bounded Euler class. 

    In \cite{BBF21}, Bestvina--Bromberg--Fujiwara show that the mapping class group of any finite-type surface has property QT. By Theorem \ref{MainThm1}, we have 
    \begin{theorem}[Theorem \ref{IntroThm: MCG}]
        $\mcg(\Sigma)$ has property QT for any finite-type surface $\Sigma$ with boundary. In particular, braid groups have property QT.
    \end{theorem}

    Now let $\Sigma$ be a finite-type surface, possibly with boundary.
    A \textit{multicurve} $C\subset \Sigma$ is a collection of pairwise disjoint, homotopically distinct, essential simple closed curves. 
    Let $C=\{\alpha_1,\dots,\alpha_n\}$ be a multicurve on $\Sigma$. Let $\mcg(\Sigma;C)<\mcg(\Sigma)$ denote the stabilizer of $C$. Let $N$ be an open neighborhood of $\bigcup_{i=1}^n \alpha_i$, and let $\{\alpha_i^+,\alpha_i^-\}$ be the boundary components of the neighborhood of $\alpha_i$. By Theorem \ref{theorem: Inclusion}, the inclusion $\Sigma-N\hookrightarrow \Sigma$ induces a central extension
    \[1\to T_{C^{\pm}}\to \mcg(\Sigma-N) \to \mcg(\Sigma;C)\to F\to 1,\]
    where $T_{C^{\pm}}$ is the free abelian group generated by $T_{\alpha_1^+}T_{\alpha_1^-}^{-1},\dots,T_{\alpha_n^+}T_{\alpha_n^-}^{-1}$ and $F$ is a finite group. By Theorem \ref{IntroThm: MCG}, $\mcg(\Sigma-N)$ has property QT. Therefore, Theorem \ref{MainThm1} and the commensurability invariance of property QT give the following.

    \begin{theorem}[Theorem \ref{IntroThm: Multicurve}]
        Let $\Sigma$ be a finite-type surface possibly with boundary. The stabilizer in $\mcg(\Sigma)$ of any multicurve on $\Sigma$ has property QT. 
    \end{theorem}

\subsection{Outer automorphism groups of hyperbolic groups}

Then we prove Theorem \ref{IntroThm: Out(G)}, which says that for any torsion-free one-ended hyperbolic group $G$, $\out(G)$ has property QT. Our proof relies on the following theorem.

\begin{theorem}\cite[Theorem 1.2]{Lev05}\label{Thm: VirtDireProd}
    Let $G$ be a torsion-free one-ended hyperbolic group. The group $\out(G)$ is virtually a direct product $\Z^q\times M$, where $M$ is the quotient of $\prod_{i=1}^m\pmcg^{\partial}(\Sigma_i)$ by a central subgroup isomorphic to $\Z^r$. 
\end{theorem}

\begin{proof}[Proof of Theorem \ref{IntroThm: Out(G)}]
    By Theorem \ref{IntroThm: MCG} and Theorem \ref{MainThm1}, the group $M$ in Theorem \ref{Thm: VirtDireProd} has property QT. Since property QT is invariant under taking direct product and finite-index supergroup, $\out(G)$ also has property QT.
\end{proof}

\section{Open questions}\label{sec: OpenQues}
In this final section, we collect some open questions about properties PH and QT. 

Corollary \ref{Cor: LinealPH} shows that the property of being finitely generated is implied by the property PH from lineal actions. Note that the essential part of the proof of Corollary \ref{Cor: LinealPH} is Lemma \ref{Lem: OriLineal}, which uses lineal actions to construct a finite generating set. However, it is not clear how to construct a finite generating set for a group with property PH from general hyperbolic actions. So we raise the following question.
\begin{question}
    Does property PH imply finite generation? 
\end{question}

If each hyperbolic space in the definition of property PH is assumed to be proper, then we say $G$ has \textit{property PPH}, which is studied by Cui and Wan in \cite{CW25}. In particular, they obtained the following Tits Alternative for finitely generated groups with property PPH: a finitely generated group with property PPH from focal and lineal actions must be amenable. We think the additional assumption that each hyperbolic space is proper is not essential. So we raise the following question. 

\begin{question}\label{Que: PHFromFocal}
    Are (finitely generated) groups with property PH from focal and lineal actions amenable?
\end{question}


Let $\L=\langle t, a, b\mid [a,b]=1, ta^2b^{-1}t^{-1}=a^2b, tab^2t^{-1}=a^{-1}b^2\rangle$ be the Leary-Minasyan group defined in \cite{LM21}. It is proved in \cite{LM21} that $\L$ is a CAT(0) group quasi-isometric to $F_5\times \Z^2$. Button \cite{But25} utilized this group to show that property QT is not a quasi-isometry invariant. As a generalization of both word-hyperbolicity and property QT, we wonder that
\begin{question}\label{Que: PHQIInv}
    Is property PH on finitely generated groups a quasi-isometry invariant? In particular, does the Leary-Minasyan group $\L$ have property PH?
\end{question}

Recall that a finitely generated group with property PH may contain distorted elements. By studying existing examples with or without property PH, we ask

\begin{question}
    Is every element in a finitely generated group with property PH either undistorted or at least exponentially distorted?
\end{question}

In \cite{WY25}, Yang and the second author obtained the locally uniform exponential growth for a large class of finitely generated groups with property PH. One application there is to recover the classical Mangahas's result \cite{Man10} that mapping class groups have locally uniform exponential growth. However, as an analogue of mapping class groups, it is still open whether $\out(F_n)$ with $n\ge 3$ has locally uniform exponential growth. Motivated by the proof outline of \cite{WY25}, we wonder

\begin{question}
    Does $\out(F_n)$ with $n\ge 3$ have property PH?
\end{question}

\begin{question}
    Does the shadowing property defined in \cite{WY25} have an invariance similar to Theorem \ref{MainThm1}?
\end{question}

\printbibliography

\end{document}